\documentclass[11pt,reqno]{amsart}
\usepackage{amsthm,amssymb,amsmath}
\usepackage{textcomp}

\usepackage[T1]{fontenc}

\newtheorem{theorem}{Theorem}[section]
\newtheorem*{theorem*}{Theorem}
\newtheorem{lemma}{Lemma}[section]
\newtheorem{corollary}{Corollary}[section]
\newtheorem{proposition}{Proposition}[section]
\newtheorem{claim}{\sc Claim}

\theoremstyle{definition}

\newtheorem{definition}{\bf Definition}[section]
\newtheorem*{definition*}{\bf Definition}

\newtheorem{example}{\bf Example}[section]

\newtheorem{remark}{Remark}
\newtheorem*{remark*}{Remark}
\newtheorem*{remarks}{Remarks}

\newtheorem*{example*}{\bf Example}

\newcommand{\loc}{{\rm loc}}
\newcommand{\supp}{{\rm sprt\,}}

\begin{document}

\title[{Stochastic equations with singular drift}]{Stochastic equations with time-dependent singular drift}

\author{D.\,Kinzebulatov}

\address{Universit\'{e} Laval, D\'{e}partement de math\'{e}matiques et de statistique, pavillon Alexandre-Vachon 1045, av. de la M\'{e}decine, Qu\'{e}bec, QC, G1V 0A6, Canada}

\email{damir.kinzebulatov@mat.ulaval.ca}

\author{K.R.\,Madou}

\address{Universit\'{e} Laval, D\'{e}partement de math\'{e}matiques et de statistique, pavillon Alexandre-Vachon 1045, av. de la M\'{e}decine, Qu\'{e}bec, QC, G1V 0A6, Canada}

\email{kodjo-raphael.madou.1@ulaval.ca}

\thanks{The research of D.K.\,is supported by the Natural Sciences and Engineering Research Council 
of Canada (grant RGPIN-2017-05567). The research of K.R.M. is supported by the Institut des sciences math\'{e}matiques, Montr\'{e}al}

\subjclass[2010]{60H10, 47D07 (primary), 35J75 (secondary)}

\keywords{Parabolic equations, Feller processes, stochastic equations, singular drifts}

\begin{abstract}
We prove unique weak solvability for 
stochastic differential equations with drift in a large class of time-dependent vector fields. This class contains, 
in particular, the critical Ladyzhenskaya-Prodi-Serrin class, the weak $L^d$ class as well as some vector fields that are not even in $L^{2+\varepsilon}_{\loc}$, $\varepsilon>0$.
\end{abstract}

\maketitle

\section{Introduction}

The subject of this paper is the problem of existence and uniqueness of a weak solution to 
stochastic  equation
\begin{equation}
\label{sde}
X_t=x-\int_{0}^{t} b(r,X_r)dr+ \sqrt{2} W_t, \quad t \geq 0,\quad  x \in \mathbb{R}^{d}, \quad d \geq 3,
\end{equation}
where $W_t$ is a $d$-dimensional Brownian motion, and the vector field  $b(t, x):[0,\infty[ \times \mathbb R^d \rightarrow \mathbb R^d$ can be singular (i.e.\,locally unbounded) both in $t$ and $x$ variables.

The problem of finding the minimal assumptions on the vector field $b$ (called drift) so that, for every $x \in \mathbb R^d$, there exists a unique weak (strong) solution to \eqref{sde}, is one of the central and classical problems in the theory of diffusion processes. The necessity to work with Brownian motion perturbed by a discontinuous drift  is dictated by applications, among them the problems of the theory of controlled diffusion processes; see also \cite{RZ} and references therein regarding connections with hydrodynamics.

The study of stochastic equations with  locally unbounded drift goes back to Portenko \cite{P} who proved existence of a unique in law weak solution to \eqref{sde} assuming that
$$|b(t,\cdot)| \in L^{p}([0,T] \times \mathbb R^d),\;\;p>d+2 \quad \text{ or } \quad |b(\cdot)| \in L^p \equiv L^p(\mathbb R^d), \;p>d.$$
In the same period of time, Kovalenko-Sem\"{e}nov \cite{KS} considered the corresponding  Kolmogorov operator $-\Delta + b \cdot \nabla$
with stationary $b=b(x)$ in a wide class  of form-bounded vector fields (see Definition \ref{def1} below)  and constructed an associated Feller semigroup. 
Their result serves as the point of departure for the present paper.
The next important step was made by Krylov-R\"{o}ckner \cite{KR} who proved that \eqref{sde} has a unique strong solution provided that
\begin{equation}
\label{LPS}
\tag{LPS}
|b| \in L_{\loc}^q([0,\infty[,L^r+L^\infty), \quad \frac{d}{r}+\frac{2}{q} < 1, \quad 2<q < \infty
\end{equation}
(Ladyzhenskaya-Prodi-Serrin class). Further, in \cite{BFGM} Beck-Flandoli-Gubinelli-Maurelli considered $b$ 
in the critical Ladyzhenskaya-Prodi-Serrin class
\begin{equation}
\label{LPS2}
\tag{${\rm LPS}_c$}
|b| \in L_{\loc}^q([0,\infty[,L^r+L^\infty) \quad \text{ for $r \geq d$, $q \geq 2$,}\quad \frac{d}{r}+\frac{2}{q} \leq 1
\end{equation}
and proved that there exists a unique strong solution to equation $dX_t=-b(t,X_t)dt + \sqrt{2}dW_t$ starting with a diffusive random variable, using an approach based on solving the corresponding to \eqref{sde} stochastic transport equation. In \cite{XXZZ}, Xia-Xie-Zhang-Zhao proved weak well-posedness of \eqref{sde}, for every initial point $x \in \mathbb R^d$, in the case $|b| \in C([0,T],L^d)$.
Recently, R\"{o}ckner-Zhao \cite{RZ} constructed a weak solution to \eqref{sde} unique in an appropriate class (i.e.\,satisfying Krylov-type estimate) for $b=b_1+b_2$, where 
\begin{equation}
\label{RZ_class}
|b_1| \in L^q\big([0,T],L^r\big),\;\;\frac{d}{r}+\frac{2}{q}=1,\;\;r \in ]d,\infty[ \qquad \text{ and }\quad b_2 \in L^\infty([0,T],L^{d,w}),
\end{equation}
where $L^{d,w}$ is the weak $L^d$ class (in fact, appropriately localized) reaching critical-order singularities.

We comment on the existing literature on \eqref{sde} in greater detail further below.

In the present paper we treat the problem of weak well-posedness of \eqref{sde} with $b$ in a large class of time-dependent vector fields. It contains the critical Ladyzhenskaya-Prodi-Serrin class \eqref{LPS2}, the class \eqref{RZ_class}, other classes of vector fields having critical-order singularities, as well as some vector fields $b=b(x)$ with $|b|$ not even in $L^{2+\varepsilon}_{\loc}$, $\varepsilon>0$.

\begin{definition}
\label{def1}
A vector field $b:[0,\infty[ \times \mathbb R^d \rightarrow \mathbb R^d$ is said to be (time-dependent) form-bounded  if $|b| \in L^2_{\loc}([0,\infty[ \times \mathbb R^d)$ and there exist a constant $\delta>0$  such that
\begin{align*}
 \int_0^\infty \|b(t,\cdot)\varphi(t,\cdot)\|_2^2 dt   \leq \delta \int_0^\infty\|\nabla \varphi(t,\cdot)\|_2^2 dt+\int_0^\infty g(t)\|\varphi(t,\cdot)\|_2^2dt
\end{align*}
for all $\varphi \in C^\infty_c([0,\infty[ \times \mathbb R^d)$,
for a non-negative function $g=g_\delta \in L^1_{\loc}([0,\infty[)$.  Here and below, $\|\cdot\|_p:=\|\cdot\|_{L^p}$.

We write $b=b(t,x) \in \mathbf{F}_\delta$.

Shortly,
$$
|b(t,\cdot)|^2 \leq \delta (-\Delta) + g(t) \quad \text{ in the sense of quadratic forms}.
$$

The constant $\delta$ is called the form-bound of $b$. It plays a fundamental role in what follows.

\end{definition}

Our main result, stated briefly, is as follows.

\begin{theorem*}
Let $d \geq 3$, $b=b(t,x) \in \mathbf{F}_\delta$, $\delta<d^{-2}$. For every $x \in \mathbb R^d$ there exists a weak solution to \eqref{sde} that is unique in an appropriate class (dependent on $b$). The weak solutions to \eqref{sde} constitute a Feller process.
\end{theorem*}

For detailed statement, see Theorem \ref{thm1}. Its proof is essentially operator-theoretic (see outline below). Our principal object is a Feller evolution family associated to the Kolmogorov operator $-\Delta + b(t,x) \cdot \nabla$, $b \in \mathbf{F}_\delta$. The stochastic equation \eqref{sde} is used  to characterize the corresponding Feller process.

\begin{example}
The following are some sub-classes of $\mathbf{F}_\delta$ defined in elementary terms. 

We have
$$
b \in \eqref{LPS2} \quad \Rightarrow \quad \text{$b \in \mathbf{F}_\delta$},
$$
$$
b \in L^\infty([0,\infty[,L^{d,w})  \quad \Rightarrow \quad b \in \mathbf{F}_\delta,
$$
$$
b \in L^\infty([0,\infty[,\mathbf{C}_s)  \quad \Rightarrow \quad b \in \mathbf{F}_\delta,
$$
where $L^{d,w}=L^{d,w}(\mathbb R^d)$ and $\mathbf{C}_s=\mathbf{C}_s(\mathbb R^d)$ ($s>1$) are the weak $L^d$ class and the Campanato-Morrey class, respectively, with $\delta$ depending on the norm of $b$ in these classes. 

We discuss these examples in detail in Section \ref{ex_sect}.

Note that
$$
b_1 \in \mathbf{F}_{\delta_1}, b_2 \in \mathbf{F}_{\delta_2} \quad \Rightarrow \quad b_1+b_2 \in \mathbf{F}_\delta, \qquad \sqrt{\delta}:=\sqrt{\delta_1}+\sqrt{\delta_2},
$$
so the sums of the vector fields from different classes listed above are form-bounded as well. 

For example,
$$
|b(t,x)|^2 \leq \delta \biggl(\frac{d-2}{2}\biggr)^2 \kappa(t) |x-x_0|^{-2} + C|t-t_0|^{-1}\bigl(\log(e+|t-t_0|^{-1}) \bigr)^{-1-\gamma}, \quad \gamma>0, \quad C \in \mathbb R,
$$
where $\kappa$ is measurable, $|\kappa(t)| \leq 1$, 
is in $\mathbf{F}_\delta$
(by Hardy's inequality). 

For every $\varepsilon>0$, there exist $b=b(x) \in \mathbf{F}_\delta$ such that $b \not \in L^{2+\varepsilon}_{\loc}$, e.g. 
$$
|b(x)|^2=C\frac{\mathbf{1}_{B(0,1+a)} - \mathbf{1}_{B(0,1-a)}}{\big| |x|-1\big|^{-1}(-\ln\big||x|-1\big|)^c}, \quad c>1, \quad 0<a<1.
$$
\end{example}

Let us emphasize that $\mathbf{F}_\delta$ is not a refinement of the class \eqref{LPS2} in the sense that $\mathbf{F}_\delta$ is not situated between $\frac{d}{r}+\frac{2}{q} =1$ and $\frac{d}{r}+\frac{2}{q}>1$. 
In contrast to the sub-classes of $\mathbf{F}_\delta$ listed above, the class $\mathbf{F}_\delta$ is defined in terms of the operators that constitute  $-\Delta + b (t,x) \cdot \nabla$.

\begin{example}
\label{ex0}
In \cite[Sect.\,7]{BFGM}, the authors show that the stochastic equation \eqref{sde} with  the initial point $x=0$ and the vector field 
$$
b(x)=\sqrt{\delta}\frac{d-2}{2} |x|^{-2}x \in \mathbf{F}_\delta,
$$
does not have a weak solution if $\delta>4(\frac{d}{d-2})^2$. Informally, the attraction to the origin by $b$ is so strong that the process get stuck there with positive probability. See also \cite{W}. 

On the other hand, if $\delta<d^{-2}$, then by Theorem \ref{thm1} a weak solution exists. (In fact, since this vector field is time-independent, a less restrictive assumption on $\delta$ would suffice, see \cite{KiS2}.)

Thus, the existence of a weak solution to \eqref{sde} depends on the value of $\delta$.

\end{example}

It should be noted that, generally speaking, additional constraints on ${\rm div\,}b$ allow to further weaken the regularity assumptions on $b$, see Zhang-Zhao \cite{ZZ}, Zhao \cite{Zh}, R\"{o}ckner-Zhao \cite{RZ} and references therein. These results, however, lie outside of the scope of the present paper, so we will not comment on them further.

\subsection*{Existing results}
After Krylov-R\"{o}ckner \cite{KR}, Zhang \cite{Z} established strong well-posedness for 
\begin{equation}
\label{sde2}
X_t=x-\int_{0}^{t} b(r,X_r)dr+ \int_0^t \sigma(r,X_r) dW_r \qquad (\text{$\sigma$ unif.\,non-degenerate})
\end{equation}
when $b \in \eqref{LPS}$ and $\sigma$ is uniformly continuous, $|\nabla \sigma| \in \eqref{LPS}$, see also \cite{Z0,Z2}.

Regarding the stationary case, in Kinzebulatov-Sem\"{e}nov \cite{KiS} the authors considered equation \eqref{sde} with $b=b(x)$ in the class of \textit{weakly} form-bounded vector fields: $|b| \in L^1_{\loc}$ and 
\begin{equation}
\tag{$\mathbf{F}_\delta^{\scriptscriptstyle 1/2}$}
\||b|^{\frac{1}{2}}\psi\|_2^2 \leq \delta \|(\lambda-\Delta)^{\frac{1}{4}}\psi\|_2^2 \quad \text{ for all } \psi \in C_c^\infty(\mathbb R^d)
\end{equation}
for some constants $\delta>0$ and $\lambda=\lambda_\delta \geq 0$. The class of form-bounded vector fields $b=b(x) \in \mathbf{F}_\delta$ is a proper subclass of $\mathbf{F}_{\delta^2}^{\scriptscriptstyle 1/2}$. 
They proved that if $\delta<c_d$ (an explicit constant), then the Feller semigroup associated to $-\Delta + b \cdot \nabla$ (see \cite{KiS4}) provides weak solution to equation \eqref{sde}, for every $x \in \mathbb R^d$. The smallness of $\delta$ is in fact necessary in view of Example \ref{ex0}. Moreover, the family of these weak solutions is ``sequentially'' unique, cf.\,remark after Theorem \ref{thm1}. 

Earlier, Bass-Chen \cite{BC} proved that there exists a unique in law weak solution to \eqref{sde} assuming that $b=b(x)$ is in the Kato class. The Kato class contains some vector fields $|b| \not \in L^{1+\varepsilon}_{\loc}$, $\varepsilon>0$, and is a proper subclass of $\mathbf{F}_{\delta}^{\scriptscriptstyle 1/2}$.

The  result of \cite{KiS} 
was extended in  \cite{KiS2} to equation \eqref{sde2} with bounded $\sigma=\sigma(x)$ satisfying $\nabla \sigma \in \mathbf{F}_{\delta_1}$, which allows to treat $\sigma$ having critical discontinuities, although at expense of restricting the class of the drifts $b=b(x)$ from $\mathbf{F}_\delta^{\scriptscriptstyle 1/2}$ to $\mathbf{F}_\delta$.

Regarding the strong well-posedness of equation \eqref{sde} with $b \in \mathbf{F}_\delta$, in Kinzebulatov-Sem\"{e}nov-Song \cite{KSS}  the result of \cite{BFGM} on the well-posedness of the stochastic transport equation  was extended in the stationary case to include drifts $b=b(x) \in \mathbf{F}_\delta$, which allows to construct a strong solution to $dX_t=-b(X_t)dt + \sqrt{2}dW_t$ with diffuse initial data arguing as in \cite{BFGM}.

In \cite{Kr0}, Krylov proved that there exists a unique strong solution to \eqref{sde2} if $|\nabla \sigma| \in L^d_{\loc}$, $|b| \in L^d$. In \cite{Kr1_}, he constructed a strong Markov process that provides a weak solution to \eqref{sde2} assuming that $\sigma=\sigma(t,x)$ is only bounded measurable, and
$$
|b| \in L_{\loc}^q([0,\infty[,L^r) \quad \text{ for $r \in [d,\infty]$, $q \in [1,\infty]$,}\quad \frac{d}{r}+\frac{1}{q} \leq 1.
$$
In \cite{Kr5,Kr5_}, the author investigated the properties of these solutions, establishing, in particular, the boundedness of the resolvent operator, It\^{o}'s formula, Harnack inequality, see also \cite{Kr1,Kr2,Kr3}. 
In \cite{Kr6}, the author considered $|b| \in L^{d+1}([0,1] \times \mathbb R^d)$ (or, more generally, in a Morrey class) and proved, in particular, that the problem
$$
(\partial_t + \Delta  + b \cdot \nabla)v = \mathsf{f}, \quad v(1,\cdot)=0,
$$
has a unique solution, and the latter satisfies, for every $p \in ]1,d+1[$, $$\|\partial_t v\|_{L^p([0,1]\times \mathbb R^d)},\;\|D^2v\|_{L^p([0,1]\times \mathbb R^d)} \leq N_1 \|\mathsf{f}\|_{L^p([0,1]\times \mathbb R^d)} + N_2\|\mathsf{f}\|_{L^q([0,1]\times \mathbb R^d)}, \quad  N_i=N_i(d,p),$$ and
$$
\|\nabla v\|_{L^q([0,1]\times \mathbb R^d)} \leq N_1\|\mathsf{f}\|_{L^{q}([0,1] \times \mathbb R^d)}, \qquad q:=\frac{p(d+1)}{d+1-p}.
$$
(The estimates of the same type  on the ``order $1+\varepsilon$'' derivatives of solution also play crucial role in the present paper, see \eqref{reg11} below.)
We also refer to recent papers \cite{DK,Kr4} where Alexandrov type estimates are obtained for drifts of Morrey type.

After \cite{RZ} where R\"{o}ckner-Zhao constructed a unique weak solution to \eqref{sde} for $b=b_1+b_2$ satisfying \eqref{RZ_class}, they proved in \cite{RZ2} the strong existence and uniqueness for \eqref{sde} assuming that  either $|b| \in C([0,T],L^d)$ or $|b| \in L^q\big([0,T],L^r\big)$, $\frac{d}{r}+\frac{2}{q}=1$, $r \in ]d,\infty[$.

The examples above show that the class \eqref{RZ_class} considered in \cite{RZ} is a proper subclass of $\mathbf{F}_\delta$. Moreover, we have from the very beginning the strong Markov property of the constructed weak solutions as a consequence of the Feller property. 
It should be added, however, that we prove uniqueness in a class of weak solutions different from the one considered in \cite{RZ}.

\medskip

\subsection*{About the proof} The main ingredients of the proof of Theorem \ref{thm1} are as follows:

(a) A Feller evolution family $\{U^{t,s}\}_{s \leq t}$ ($\equiv$ contraction positivity preserving strongly continuous evolution family of bounded linear operators on the space $C_\infty:=\{f \in C(\mathbb R^d) \mid \lim_{x \rightarrow \infty}f(x)=0\}$ endowed with the $\sup$-norm) such that the function 
 $$u(t,\cdot):=U^{t,s}f(\cdot), \quad f \in C_\infty$$ is a weak solution to Cauchy problem
$
(\partial_t - \Delta + b (t,x)\cdot \nabla)u=0$, $u(s,\cdot)=f(\cdot).
$
The Feller evolution family $\{U^{t,s}\}_{s \leq t}$ is constructed using an approximation of $b$ by smooth bounded vector fields $b_m$ that do not increase the form-bound $\delta$ of $b$ (see Theorem \ref{thm1}(\textit{i}) below for detailed statement). 

(b) The estimate 
\begin{align}
\label{reg11}
\|\nabla v\|^q_{L^\infty([s,t],L^q)} & + \|\nabla|\nabla v|^{\frac{q}{2}}\|_{L^2([s,t],L^2)}^2 \notag \\
& \leq C \bigl(\|\mathsf{f} |h|^{\frac{q}{2}}\|^2_{L^2([s,t],L^2)} + \|\nabla f\|^q_{q}\big), \qquad q \in ]d,\delta^{-\frac{1}{2}}[ 
\end{align}
for the solution $v$ to the inhomogeneous Cauchy problem
\begin{equation*}
(\partial_t - \Delta + b(t,x) \cdot \nabla)v=|\mathsf{f}|h, \quad 
v(s,\cdot)=f(\cdot) \in W^{1,q}
\end{equation*}
where $\mathsf{f} \in \mathbf{F}_\beta$, $\beta<\infty$, $h$ is bounded and has compact support, with constant $C$ independent of $h$ and $f$. 

An estimate of the type \eqref{reg11} appeared for the first time in the fundamental paper of Kovalenko-Sem\"{e}nov \cite{KS}. There the authors proved that the solution $w$ to the elliptic equation $$(\mu - \Delta + b \cdot \nabla)w=f, \quad b=b(x) \in \mathbf{F}_\delta, \quad \delta < 1 \wedge (\frac{2}{d-2})^2$$ satisfies 
\begin{equation}
\label{reg12}
\|\nabla w\|_q + \|\nabla |\nabla w|^{\frac{q}{2}}\|^{\frac{2}{q}}_2 \leq C \|f\|_q, \quad q \in ]2 \vee (d-2),\frac{2}{\sqrt{\delta}}[, \quad \mu>\mu_0>0,
\end{equation}
as was needed to carry out an $L^p \rightarrow L^\infty$ iteration procedure that verifies conditions of the Trotter Approximation Theorem in $C_\infty$. The latter yields the corresponding to $-\Delta + b \cdot \nabla$ Feller semigroup.

The Feller evolution family $\{U^{t,s}\}$, employed in the present paper, was constructed in \cite{Ki} via a direct (parabolic) variant of the iteration procedure of \cite{KS}, which we outline below in Section \ref{proof_i_sect}. The regularity estimate \eqref{reg11} is the content of Theorem \ref{prop2_} below. 
We note that there is a non-negligible difference between the proofs of \eqref{reg11} and \eqref{reg12} due to presence of the term $\partial_t v$ in the former, which forces the more restrictive assumption $\delta<d^{-2}$ (compared to $\delta < 1 \wedge (\frac{2}{d-2})^2$ in \cite{KS}, see Remark \ref{el_par} below).

Armed with (a), (b), we provide two constructions of the weak solution to \eqref{sde}. The first construction follows \cite{KiS}, \cite{KiS2} and uses as the point of departure the probability measures on the c\`{a}dl\`{a}g trajectories determined by the Feller evolution family $U^{t,s}$. The second construction uses a tightness argument, similarly to the proof of the existence in \cite{RZ}. In view of the approximation result for $U^{t,s}$, thus constructed probability measures coincide with the ones in the first proof. 

The proof of the uniqueness also appeals to the approximation result for $\{U^{t,s}\}$ and the regularity estimate \eqref{reg11}. This is in addition to the fact that $\{U^{t,s}\}$ (and thus the family of the weak solutions to \eqref{sde} parametrized by $x \in \mathbb R^d$) does not depend on the choice of a bounded smooth approximation $\{b_m\}$ of $b$ preserving the form-bound of $\delta$, see above. Similarly to \cite{KiS}, \cite{KiS2}, we consider the latter to be a uniqueness result for \eqref{sde} on its own.

\subsection*{Further discussion}
1.~Generally speaking, a drift $b \in \mathbf{F}_\delta$ rules out $W^{2,p}$ estimates on solutions to the corresponding elliptic or parabolic equation for $p$ large. More precisely, let $(\mu -\Delta  + b \cdot \nabla )w=f$, $b=b(x) \in \mathbf{F}_\delta$, $f \in C_c^\infty$ and $w \in W^{1,r}$ for $r$ large (e.g.\,by \eqref{reg12}). Then, taking into account that for every $\varepsilon>0$ there exist $b \in \mathbf{F}_\delta$ such that $b \not \in L^{2+\varepsilon}_{\loc}$, one only has
$$
\Delta w = \lambda w + b \cdot \nabla w \in L^{\frac{2r}{2+r}}_{\loc}$$
(this is in contrast to the sub-class $|b| \in L^d$, which provides $\Delta w \in L^{\frac{dr}{d+r}}$).

\smallskip

2.~The heat kernel of $-\Delta + b(t,x) \cdot \nabla$, $b \in \mathbf{F}_\delta$ does not satisfy in general neither upper nor lower Gaussian bound. In fact, already for  $b(x)=c|x|^{-2}x$ the sharp two-sided bounds on the heat kernel take form ``a Gaussian density multiplied by a singular weight if $c>0$, or a vanishing weight if $c<0$''. Nevertheless, the two-sided Gaussian bounds on the heat kernel of $-\Delta + b(t,x) \cdot \nabla$ are valid when $b \in \mathbf{F}_\delta$ but under additional assumptions on ${\rm div\,}b$, such as the Kato class condition, see details in Kinzebulatov-Sem\"{e}nov \cite{KiS4} (in a more general context of divergence-form parabolic equations).

\smallskip

3.~The proofs of the main results of the present paper (Theorem \ref{thm1} and Theorem \ref{prop2_} below) can be extended, arguing as in \cite{KiS2}, to stochastic equation 
\begin{equation}
\label{sde_sigma}
X_t=x-\int_{0}^{t} b(r,X_r)dr+ \sqrt{2}\int_0^t \sigma(r,X_r) dW_r, \quad x \in \mathbb R^d,
\end{equation}
with $b \in \mathbf{F}_\delta$ and $\sigma:[0,\infty[ \times \mathbb R^d \rightarrow \mathbb R^d \otimes \mathbb R^d$ (measurable, bounded, uniformly non-degenerate) such that $a:=\sigma\sigma^{\scriptscriptstyle \top}$ satisfies
\begin{equation}
\label{nondiv}
\big(\partial_{x_k}a_{ij}\big)_{i=1}^d \in \mathbf{F}_{\gamma_{kj}}, \quad 1 \leq j,k \leq d,
\end{equation}
assuming that the form-bounds $\delta$ and $\gamma_{kj}$ are smaller than certain explicit constants dependent on the dimension $d$. The latter allows to treat form-bounded drifts together with some diffusion coefficients having critical discontinuities. For instance, 
$$
a(t,x)=I+ \kappa(t)\frac{x \otimes x}{|x|^2},
$$
where $\kappa \in L^\infty([0,\infty[)$ is measurable, $\inf_{t \geq 0}\kappa(t)>-1$ and $\|\kappa\|_{L^\infty([0,\infty[)}$ is sufficiently small, or an infinite series of such matrices discontinuous at different points. Another example is
$$
a(t,x)=I+c_1\sin^2 \big(\kappa(t)\log|x|\big) e_1 \otimes e_1 + c_2\sin^2 \big(t^{-\alpha}|x|^{1-\frac{1}{\beta}}\big)e_2 \otimes e_2,
$$
where $e_i \in \mathbb R^d$, $|e_i|=1$, $c_i>0$ ($i=1,2$), $\|\kappa\|_{L^\infty([0,\infty[)}$ is sufficiently small, and $\beta>1$, $\alpha<\frac{\beta-1}{2\beta}$.

Let us note that condition \eqref{nondiv} arises first of all as the condition providing the
Sobolev regularity estimate \eqref{reg11} for solutions to the divergence-form parabolic equation $(\partial_t -\nabla \cdot a \cdot \nabla + b \cdot \nabla)v=0$, $b \in \mathbf{F}_\delta$. The second use for \eqref{nondiv} is to put the corresponding to \eqref{sde_sigma}  Kolmogorov operator $-a \cdot \nabla^2 + b \cdot \nabla$ in divergence form $-\nabla \cdot a \cdot \nabla + \hat{b} \cdot \nabla$, $\hat{b}:=\nabla a + b \in \mathbf{F}_\delta$, as is needed to apply the regularity estimate to construct the Feller evolution family and the weak solution to \eqref{sde_sigma}. See \cite{KiS2} for detailed discussion.

\smallskip

4.~The result of Kovalenko-Sem\"{e}nov \cite{KS} (i.e.\,estimate \eqref{reg12} + Feller semigroup)  has an alternative, somewhat more elementary proof \cite{Ki9}. 
Namely, under the same assumption on $\delta$ as in \cite{KS}, one  first ``guesses'' the resolvent $R_q(\mu)$  of $-\Delta + b \cdot \nabla$ in $L^q$, $q \in ]2 \vee (d-2),\frac{2}{\sqrt{\delta}}[$. That is, for $r$, $p$ such that $2 \leq r<q<p<\infty$, set
\begin{equation}
\label{res}
\tag{$\star$}
R_q(\mu):=(\mu - \Delta)^{-1} - (\mu - \Delta)^{-\frac{1}{2}-\frac{1}{p}} Q_{q}(p) (1 + T_q)^{-1} G_{q}(r) (\mu - \Delta)^{-\frac{1}{2}+\frac{1}{r}}, \quad \mu>\mu_0,
\end{equation}
where the operators in $L^q$
$$
G_q(r) \equiv b^{\frac{2}{q}} \cdot \nabla (\mu -\Delta )^{-\frac{1}{2}-\frac{1}{r}}, \quad 
Q_q(p) \equiv (\mu -\Delta )^{-\frac{1}{2}+\frac{1}{p}}|b|^{1-\frac{2}{q}}, \quad b^{\frac{2}{q}}:=|b|^{\frac{2}{q}-1}b,
$$
$$
T_q \equiv b^{\frac{2}{q}}\cdot \nabla(\mu - \Delta)^{-1}|b|^{1-\frac{2}{q}}
$$
are bounded and $\|T_q\|_{q \rightarrow q}<1$ ($L^q\rightarrow L^q$ norm). Here the boundedness of $G_q(r) $, $Q_q(p)$, $T_q$ is a consequence of the hypothesis $b \in \mathbf{F}_\delta$, while $\|T_q\|_{q \rightarrow q}<1$ follows from the assumption on $q$ and $\delta$. In fact, expanding $(1+T_q)^{-1}$ in the geometric series, one obtains that the RHS of \eqref{res} coincides with the formal K.\,Neumann series for $\mu-\Delta + b \cdot \nabla$.
Now, since $q>d-2$, the regularizing factor $(\mu - \Delta)^{-1/2-1/p}$ in \eqref{res}, with $p$ chosen sufficiently close to $q$, yields, via the Sobolev Embedding Theorem, that, for every $f \in L^q$, the solution $w=R_q(\mu)f$ to the equation $(\mu-\Delta + b \cdot \nabla)w=f$ is H\"{o}lder continuous. This observation allows to constructs the resolvent $R_{C_\infty}(\mu)$ of the sought Feller generator by the formula
$$
R_{C_\infty}(\mu):=\bigl[R_q(\mu) \upharpoonright L^q \cap C_\infty \bigr]^{\rm clos}_{C_\infty \rightarrow C_\infty} \quad (\text{closure of operator}).
$$
Let us note that the representation \eqref{res} provides  more detailed
information about the regularity of higher-order derivatives of $w$: $(\mu - \Delta)^{\frac{1}{2}+\frac{1}{p}}w \in L^q$, cf.\,\eqref{reg12}.
That being said, the construction of the Feller semigroup via the iteration procedure of Kovalenko-Sem\"{e}nov \cite{KS} has some crucial advantages: it admits extension to time-dependent drifts (see \cite{Ki} and the present paper) and to discontinuous diffusion coefficients \cite{KiS2}.

\subsection*{Notations}Let $\mathcal B(X,Y)$ denote the space of bounded linear operators between Banach spaces $X \rightarrow Y$, endowed with the operator norm $\|\cdot\|_{X \rightarrow Y}$. $\mathcal B(X):=\mathcal B(X,X)$.

We write $T=s\mbox{-} X \mbox{-}\lim_n T_n$ for $T$, $T_n \in \mathcal B(X,Y)$ if $$\lim_n\|Tf- T_nf\|_Y=0 \quad \text{ for every $f \in X$}.
$$ 

Let $\|\cdot\|_p:=\|\cdot\|_{L^p}$.

Put
$$
\langle f,g\rangle = \langle f g\rangle :=\int_{\mathbb R^d}f gdx$$ 
(all functions considered below are real-valued).

Let
$
\|\cdot\|_{p \rightarrow q}=\|\cdot\|_{L^p \rightarrow L^q}.
$

$C_\infty:=\{f \in C(\mathbb R^d) \mid \lim_{x \rightarrow \infty}f(x)=0\}$ (with the $\sup$-norm).

$\mathcal S$ is the L.\,Schwartz' space of test functions.

We denote by $\upharpoonright$ the restriction of an operator (or a function) to a subspace (a subset).

We write $c \neq c(n)$ to emphasize that  $c$ is independent of $n$.

\medskip

\noindent\textbf{Acknowledgements.} The authors thank Yu.A.\,Sem\"{e}nov for a number of valuable comments. We thank G.\,Zhao for pointing out an error in the examples in the first version of this paper. The first author is also grateful to N.V.\,Krylov for helpful and inspiring discussions.

\medskip

\bigskip

\tableofcontents

\section{Main results}

We first introduce few notations and recall some standard definitions.

\medskip

1.~In what follows, given a form-bounded vector field $b=b(t,x) \in \mathbf{F}_\delta$, we denote by $\{b_m\}$ a sequence of smooth bounded vector fields such that
\begin{equation}
\label{b_m_1}
b_m \rightarrow b \quad \text{ in $L^2_{\loc}([0,\infty[ \times \mathbb R^d, \mathbb R^d$}) 
\end{equation}
and
\begin{equation}
\label{b_m_2}
\int_0^\infty \|b_{m}(t,\cdot)\varphi(t,\cdot)\|_2^2 dt \leq \delta\int_0^\infty\|\nabla \varphi(t,\cdot)\|_2^2 dt+\int_0^\infty g(t)\|\varphi(t,\cdot)\|_2^2dt
\end{equation}
for all $\varphi \in C_c^\infty([0,\infty[ \times \mathbb R^d)$, for a function $0 \leq g \in L^1_{\loc}([0,\infty[)$ independent of $m$ (in other words, $b_m$ do not increase the form-bound $\delta$ of $b$). An example of such $\{b_m\}$ is given in Section \ref{rem_b_m} below.

\medskip

2.~Consider Cauchy problem ($s \geq 0$)
\begin{align}
\label{CP}
\tag{${\rm CP}_b$}
\left\{
\begin{array}{l}
\bigl(\partial_t - \Delta + b(t,x)\cdot \nabla \bigr)u=0, \qquad (t,x) \in ]s,\infty[ \times \mathbb R^d, \notag \\[2mm]
 u(s+,\cdot)=f(\cdot).
\end{array}
\right.
\end{align}

\begin{definition}
\label{def_w}
Let $b \in \mathbf{F}_\delta$ and $f \in L^2_{\loc}$.
A real-valued function $u$ on $]s,\infty[ \times \mathbb R^d$ is called a weak solution to \eqref{CP}
if

1) $u \in L^\infty_{\loc}\bigl(]s,\infty[,L^2_{\loc}\bigr)$,

2) $|\nabla u| \in L^2_{\loc}\bigl(]s,\infty[,L^2_{\loc}\bigr)$, so $b \cdot \nabla u \in L^1_{\loc}\bigl(]s,\infty[,L^1_{\loc}\bigr)$,

3) the integral identity
\begin{equation*}
\int_{0}^\infty \langle u,\partial_t h\rangle dt = \int_s^\infty \bigl(\langle  \nabla u, \nabla h\rangle + \langle \nabla u, bh\rangle\bigr)dt
\end{equation*}
is valid for all $h \in C_c^\infty(]s,\infty[ \times \mathbb R^d)$, and
\begin{equation*}
\exists \quad \lim_{t \downarrow s} \langle u(t,\cdot),\psi(\cdot)\rangle=\langle f, \psi\rangle
\end{equation*}
for all $\psi \in L^2$ having compact support.
\end{definition}

3.~Set 
\begin{align*}
\rho(x)  \equiv \rho_{\kappa, \theta}(x) 
:=(1+\kappa |x|^2)^{-\theta}, \quad x \in \mathbb R^d,
\end{align*}
where $\theta>\frac{d}{2}$ is fixed, and $\kappa>0$ is to be chosen.
We have
\begin{equation}
\label{two_est}
|\nabla \rho| \leq \theta\sqrt{\kappa}\rho.
\end{equation}
We will be applying \eqref{two_est} to $\rho$ with $\kappa$ chosen sufficiently small.

\medskip

\begin{theorem}
\label{thm1} Let $d \geq 3$, $b \in \mathbf{F}_\delta$, $\delta<d^{-2}$. The following is true.

{\rm (\textit{i})} Let $\{b_m\}$ be bounded smooth vector fields satisfying \eqref{b_m_1}, \eqref{b_m_2}. Then the limit
$$
s\mbox{-}C_\infty\mbox{-}\lim_{m \rightarrow \infty}U^{t,s}(b_m) \quad \text{{\rm(}loc.\,uniformly in $0 \leq s \leq t<\infty${\rm)}},
$$
where $U^{t,s}(b_m)$ is the Feller evolution family of {\rm(}${\rm CP}_{b_m}${\rm)} {\rm(}Definition \ref{def_fel}{\rm)}, exists and determines a Feller evolution family, say, $U^{t,s} \equiv U^{t,s}(b)$. 
For every $f \in C_\infty$, the function $$u(t,\cdot):=U^{t,s}f(\cdot),$$ is a weak solution to \eqref{CP}. If $f \in C_\infty \cap L^2$, then $u$ is unique in the class $C([0,\infty[,L^2)$.

\medskip

{\rm (\textit{ii})} The corresponding backward Feller evolution family $P^{0,t}$, $t \in [0,T]$, $T>0$  {\rm(}Definition \ref{def_fel2}{\rm)} determines probability measures $\{\mathbb P_x\}_{x \in \mathbb R^d}$ 
on $(C([0,T],\mathbb R^d),\mathcal B_t=\sigma(\omega_r \mid 0\leq r \leq t), t \in [0,T])$, where $\omega_t$ is the coordinate process,
$$
P^{0,t}f(x) = \mathbb E_{\mathbb P_x}[f(\omega_t)], \quad t \in [0,T], \quad f \in C_\infty,
$$
such that, for every $x \in \mathbb R^d$,  $\mathbb P_x$ is 
a weak solution
to  stochastic equation 
\begin{equation}
\label{seq}
\tag{SE}
X_t=x-\int_0^t b(r,X_r)dr + \sqrt{2}W_t
\end{equation} 
 {\rm(}Definition \ref{weak_sol}{\rm)}.

For every $x \in \mathbb R^d$, $q \in ]d,\delta^{-\frac{1}{2}}[$, $\mathsf{f} \in \mathbf{F}_\beta,\;\beta<\infty$, and $ h \in C([0,T], \mathcal S)$ there exists a constant $c$ dependent only on $d$, $q$, $\delta$, $g_\delta$, $\beta$, $g_\beta$ and $T$ such that
\begin{equation}
\label{kr_est0}
\mathbb E_{\mathbb P_x}\int_0^T |\mathsf{f}(r,\omega_r)h(r,\omega_r)|dr \leq c\|\mathsf{f}|h|^{\frac{q}{2}}\|^{\frac{2}{q}}_{L^2([0,T] \times \mathbb R^d)}.
\end{equation}
Moreover, \eqref{kr_est0} holds with the RHS replaced by $c\|\mathsf{f}|h|^{\frac{q}{2}}\sqrt{\rho_x}\|^{\frac{2}{q}}_{L^2([0,T] \times \mathbb R^d)}$, where $\rho_x(y):=\rho(y-x)$, for $\kappa>0$ chosen sufficiently small.

\medskip

\medskip

{\rm (\textit{iii})} The probability measures $\{\mathbb P_x\}_{x \in \mathbb R^d}$ do not depend on the choice of $\{b_m\}$ in {\rm (\textit{i})}.

Also, if, for some $x \in \mathbb R^d$, $\mathbb P^1_x$, $\mathbb P^2_x$ are weak solution to \eqref{seq} that satisfy \eqref{kr_est0} for some $q \in ]d,\delta^{-\frac{1}{2}}[$ with $\mathsf{f}=1$ and with $\mathsf{f}=b$,
then
$$\mathbb P^1_x=\mathbb P_x^2=\mathbb P_x.$$

\end{theorem}

\begin{remarks}

1.~The following ``sequential uniqueness''  result was proved in \cite{KiS}, see also \cite{KiS2}. Let $b=b(x)$ be form-bounded (in fact, it can be weakly form-bounded, see the introduction).
Provided that the form-bound $\delta$ of $b$ is smaller than a certain explicit constant $c=c(d)$, if $\{\mathbb Q_x\}_{x \in \mathbb R^d}$ are weak solutions to \eqref{seq}, and are obtained via some approximation procedure, i.e.\,there exist
bounded smooth $\tilde{b}_n \in \mathbf{F}_\delta$ with $g$ independent of $n=1,2,\dots$ such that for every $x \in \mathbb R^d$
$$
\mathbb Q_x=w{\mbox-}\lim_{n}\mathbb P_x(\tilde{b}_n)
$$
(no convergence of $\tilde{b}_n$ to $b$ is assumed),
then 
$$
\mathbb Q_x=\mathbb P_x, \quad x \in \mathbb R^d,
$$
where $\{\mathbb P_x\}_{x \in \mathbb R^d}$ are from Theorem \ref{thm1}(\textit{ii}). We expect that a similar result can be proved in the assumptions of Theorem \ref{thm1}, but we do not address this matter here.

\smallskip

2.~The following result is a consequence of \cite[Theorem 1.1]{Se} (for the stationary case $b=b(x)$, see \cite{KS}). If $b \in \mathbf{F}_\delta$, $0<\delta<4$, then for every $p \in I_c$, where $I_c:=]\frac{2}{2-\sqrt{\delta}},\infty[$, the limit
$$
s\mbox{-}L^p\mbox{-}\lim_n U^{t,s}(b_m) \quad (\text{loc.\,uniformly in $0 \leq s \leq t <\infty$})
$$
exists and determines a strongly continuous quasi contraction positivity preserving evolution family in $L^p$, say, $U_p^{t,s}$. For every $f \in L^p$, the function $u(t,\cdot)=U_p^{t,s}f(\cdot)$ is a weak solution to \eqref{CP} \textit{in} $L^p$ (if $\delta>1$, then $p>2$, so Definition \ref{def_w} has to be modified accordingly, see \cite{Se}; it should be noted here that the interval $I_c$ is sharp, see \cite{KiS4} for detailed discussion). For all $0 \leq s < t \leq T$,
\begin{equation}
\label{pq_est}
\|u(t,\cdot)\|_q \leq C_T (t-s)^{-\frac{d}{2}(\frac{1}{p}-\frac{1}{q})}\|f\|_p, \quad \frac{2}{2-\sqrt{\delta}}<p < q<\infty
\end{equation}
and
$$
\|u(t,\cdot)\|_p \leq e^{\frac{1}{p\sqrt{\delta}}\int_s^t g(\tau)d\tau} \|f\|_p, \quad 0 \leq s \leq t < \infty
$$
(the latter easily yields that the initial condition for $u$ is satisfied in the strong sense).

By construction, $$U_p^{t,s} \upharpoonright L^p \cap C_\infty = U^{t,s} \upharpoonright L^p \cap C_\infty,$$ where $U^{t,s}$ is the Feller evolution family from Theorem \ref{thm1}(\textit{i}). 
Thus, in view of  \eqref{pq_est}, by Dunford's Theorem, $U_p^{t,s}$ and $U^{t,s}$, $0 \leq s<t<\infty$ are integral operators whose integral kernels coincide a.e.

Let us also note that $\delta=4$ is the critical value for the weak solvability of \eqref{seq} with $b(x)=\delta \frac{d-2}{2}|x|^{-2}x$ (at least as $d \rightarrow \infty$), see Example \ref{ex0} above.
\end{remarks}

The following are the key analytic results used in the proof of Theorem \ref{thm1}. 

Let $L^q_\rho:=L^q(\mathbb R^d,\rho dx)$, where $\rho(x)=(1+\kappa |x|^2)^{-\theta}$ was introduced above.

\begin{proposition}
\label{cl1}
Let $d \geq 3$, $b \in \mathbf{F}_\delta$ with $\delta<4$, $\mathsf{f} \in \mathbf{F}_\beta$ with $\beta<\infty$, both $b$ and $\mathsf{f}$ are bounded and $C^\infty$ smooth. Let $h \in C([0,T], \mathcal S)$, $ f \in C_c^\infty(\mathbb R^d)$. Fix $T>s$. Let $v$ be the solution to Cauchy problem 
\begin{equation}
\label{Cauchy3}
\left\{
\begin{array}{l}
(\partial_t - \Delta + b \cdot \nabla)v=|\mathsf{f}|h, \quad t \in [s,T] \\
v(s,\cdot)=f, 
\end{array}
\right.
\end{equation}
For every $q>\frac{2}{2-\sqrt{\delta}}$, $q \geq 2$,
there exist  constants $K$ and $\kappa$ dependent only on $d$, $q$, $\delta$, $g_\delta$, $\beta$, $g_\beta$ and $T$ such that
\begin{equation*}
\|v\|_{L^\infty([s,r],L_\rho^q)} ^q
\leq K \bigl( \|\mathsf{f} |h|^{\frac{q}{2}}\|^2_{L^2([s,r],L_\rho^2)} + \|f\|^q_{L^q_\rho}\bigr),
\end{equation*}
for $0 \leq s \leq r \leq T$. 
\end{proposition}

\begin{definition}
We say that a constant is generic if it depends only on $d$, $q$, $\delta$, $g_\delta$, $\beta$, $g_\beta$, $T$, $\rho$.
\end{definition}

\medskip

\begin{theorem}
\label{prop2_}
In the assumptions of Proposition \ref{cl1}, assume additionally that $\delta<d^{-2}$. Let $q \in ]d,\delta^{-\frac{1}{2}}[$. Then there exist generic constants $C$ and $\kappa$ such that, for all $0 \leq s \leq r \leq T$, the solution $v$ to \eqref{Cauchy3}
satisfies the following estimate 
\begin{align*}
\|v\|_{L^\infty([s,r],L_\rho^q)}^q  + \|\nabla v\|^q_{L^\infty([s,r],L_\rho^q)} & + \|\nabla|\nabla v|^{\frac{q}{2}}\|_{L^2([s,r],L_\rho^2)}^2 \\
& \leq C \bigl(\|\mathsf{f} |h|^{\frac{q}{2}}\|^2_{L^2([s,r],L_\rho^2)} + \|\nabla f\|^q_{L_\rho^q} + \|f\|^q_{L^q_\rho}\big)
\end{align*}
\end{theorem}

\bigskip

\section{Examples}

\subsection{Form-bounded vector fields}

\label{ex_sect}

 1.~One has $$
b \in L^\infty([0,\infty[,L^d+L^\infty) \quad \Rightarrow \quad b \in \mathbf{F}_\delta
$$
with appropriate $\delta$. 

Indeed, let $b=b_1+b_2$, where $b_1 \in L^\infty([0,\infty[,L^d)$, $b_2 \in L^\infty([0,\infty[,L^\infty)$. Then, by H\"{o}lder's inequality, for a.e.\,$t \in [0,\infty[$ and all $\psi \in C_c^\infty(]0,\infty[ \times \mathbb R^d)$,
\begin{align*}
\|b(t,\cdot)\psi(t,\cdot)\|_2^2 
& \leq (1+\varepsilon)\|b_1(t,\cdot)\|_d^2 \|\psi(t,\cdot)\|_{\frac{2d}{d-2}}^2 + (1+\varepsilon^{-1})\|b_2(t,\cdot)\|_\infty^2 \|\psi(t,\cdot)\|_2^2 \qquad (\varepsilon>0)\\
& (\text{we are applying the Sobolev Embedding Theorem}) \\
& \leq C_S (1+\varepsilon) \|b_1(t,\cdot)\|_d^2 \|\nabla \psi(t,\cdot)\|_2^2 + (1+\varepsilon^{-1})\|b_2(t,\cdot)\|_\infty^2 \|\psi(t,\cdot)\|_2^2.
\end{align*}
Integrating this inequality in time, we obtain that 
$b \in \mathbf{F}_{\delta}$ with $\delta=C_S(1+\varepsilon)\|b_1\|^2_{L^\infty([0,\infty[,L^d)}$.


\medskip

2.~Also,$$b \in C([0,\infty[,L^d+L^\infty) \quad \Rightarrow \quad b \in \mathbf{F}_\delta \quad \text{ with $\delta$ that can be chosen arbitrarily small.}
$$
Without loss of generality, $b \in C([0,\infty[,L^d)$.
Consider first a $b=b(x)$. Since $|b| \in L^d$, for every $\varepsilon>0$ one can represent $b=b_1+b_2$, where $\|b_1\|_d<\varepsilon$ and $\|b_2\|_\infty<\infty$. (For instance, $b_2=b\mathbf{1}_{|b| \leq m}$ and $b_1=b-b_2$, so by the Dominated Convergence Theorem $\|b_1\|_d$ can be made arbitrarily small by selecting $m$ sufficiently large.) Now the previous example applies and yields the required. In the general case, the continuity of $b$ in time allows to represent
$b(t,\cdot)=b_1(t,\cdot)+b_2(t,\cdot)$, where $\|b_1(t,\cdot)\|_d <\varepsilon$ for all $t \in [0,1]$ and 
$b_2$ is bounded on $[0,1] \times \mathbb R^d$. Repeating this on every interval $[n,n+1]$ ($n \geq 1$), one obtains $\|b_1\|_{L^\infty([0,\infty[,L^d)}<\varepsilon$ and $b_2 \in L^\infty_{\loc}([0,\infty[,L^\infty)$.
(The continuity in time is not necessary for the smallness of $\delta$, e.g.\,consider $b(t,x)=a(t)b_0(x)$ where $a \in L^\infty([0,\infty[)$ is discontinuous and $|b_0| \in L^d$.)

\medskip

3.~Further, any vector field $$
b \in L^p([0,\infty[,L^q), \quad \frac{d}{q}+\frac{2}{p} \leq 1,
$$
is in $\mathbf{F}_\delta$ with appropriate $\delta$. 

Indeed, e.g.\,in the more difficult case $\frac{d}{q}+\frac{2}{p} = 1$, by Young's inequality,
\begin{align*}
|b(t,x)|  =\frac{|b(t,x)|}{\langle |b(t,\cdot)|^q\rangle^{\frac{1}{q}}}\langle |b(t,\cdot)|^q\rangle^{\frac{1}{q}} \leq \frac{d}{q} \biggl(\frac{|b(t,x)|^q}{\langle |b(t,\cdot)|^q\rangle} \biggr)^{\frac{1}{d}} + \frac{2}{p}\bigl( \langle |b(t,\cdot)|^q\rangle^{\frac{1}{q}} \bigr)^{\frac{p}{2}},
\end{align*}
where the first term is in $ L^\infty([0,\infty[,L^d)$ (and so by the first example it is form-bounded) and the second term is in $L^2([0,\infty[,L^\infty)$ (the second term squared is to be absorbed by the function $g$). 

(If $p<\infty$, $q>d$, then one can argue as in the previous example to show that $\delta$ can be chosen arbitrarily small.)

\medskip

3.~The class of form-bounded vector fields $\mathbf{F}_\delta$ contains vector fields $b=b(x)$ with $|b|$ in $L^{d,w}$ (the weak $L^d$ class). Recall that a function $h:\mathbb R^d \rightarrow \mathbb R$ is in $L^{d,w}$ if $$\|h\|_{d,w}:=\sup_{s>0}s|\{x \in \mathbb R^d: |h(x)|>s\}|^{1/d}<\infty.$$ By the Strichartz inequality with sharp constant \cite[Prop.~2.5, 2.6, Cor.~2.9]{KPS}, if $|b|$ in $L^{d,w}$, then 
$b \in \mathbf{F}_{\delta_1}$ with 
\begin{align*}
\sqrt{\delta_1}&=\||b| (\lambda - \Delta)^{-\frac{1}{2}} \|_{2 \rightarrow 2} \\ & \leq
\|b\|_{d,w} \Omega_d^{-\frac{1}{d}} \||x|^{-1} (\lambda - \Delta)^{-\frac{1}{2}} \|_{2 \rightarrow 2}  \leq \|b\|_{d,w} \Omega_d^{-\frac{1}{d}} \frac{2}{d-2},
\end{align*}
where $\Omega_d=\pi^{\frac{d}{2}}\Gamma(\frac{d}{2}+1)$ is the volume of  $B(0,1) \subset \mathbb R^d$.

\medskip

4.~The Chang-Wilson-Wolff class $\mathbf{W}_s$ ($s>1$) consists of the vector fields $b=b(x)$
such that 
$$
|b|^2 \in L_{\loc}^s \quad \text{ and } \quad \|b\|_{W_s}:=\sup_Q \frac{1}{|Q|}\int_Q |b(x)|^2\, l(Q)^2 \varphi\big(|b(x)|^2\,l(Q)^2 \big) dx<\infty,
$$
where $|Q|$ and $l(Q)$ are the volume and the side length of a cube $Q$, respectively,
$\varphi:[0,\infty[ \rightarrow [1,\infty[$ is an increasing function such that
$
\int_1^\infty \frac{dx}{x\varphi(x)}<\infty.
$
By \cite{CWW}, $$b \in \mathbf{W}_s \quad \Rightarrow \quad b \in \mathbf{F}_{\delta}$$ with $\delta=\delta\big(\|b^2\|_{W_s}\big)<\infty$. 

The class $\mathbf{W}_s$ contains, in particular, the Campanato-Morrey class $\mathbf{C}_s$ ($s>1$):
$$
|b|^2 \in L_{\loc}^s \quad \text{ and } \quad \biggl(\frac{1}{|Q|}\int_Q |b(x)|^{2s} dx \biggr)^{\frac{1}{s}} \leq c_s l(Q)^{-2} \text{ for all cubes $Q$}.
$$

More generally, vector fields in $L^\infty([0,\infty[,\mathbf{C}_s)$ or $L^\infty([0,\infty[,\mathbf{W}_s)$ are form-bounded.

We refer to \cite{KiS4} for further discussion of form-bounded vector fields and their role in the theory of divergence-form elliptic and parabolic equations.

\subsection{Approximation}
\label{rem_b_m}
Given a $b \in \mathbf{F}_\delta$, a sequence $\{b_m\}$ of bounded smooth vector fields satisfying \eqref{b_m_1}, \eqref{b_m_2} can be constructed e.g.\,by the formula (extending $b$ to $t<0$ by $0$)
$$
b_m:=c_m e^{\varepsilon_m \Delta_{(t,x)}} (\mathbf 1_m b), 
$$
where $\mathbf 1_m$ is the indicator of $\{(t,x) \mid |b(t,x)| \leq m, |x| \leq m, |t| \leq m\}$ and $\Delta_{(t,x)}$ is the Laplace operator on $\mathbb R \times \mathbb R^d$, for appropriate $\varepsilon_m \downarrow 0$ and $c_m \uparrow 1$ so that $b_m \in \mathbf{F}_\delta$ with $g$ independent of $m$. 

\begin{proof} Set
$
\tilde{b}_m=e^{\varepsilon_m \Delta_{(t,x)}} (\mathbf 1_m b)
$
and write
$$
\tilde{b}_m=\mathbf{1}_m b + \big(\tilde{b}_m - \mathbf{1}_m b\big).
$$
Clearly, the first term $\mathbf{1}_m b \in \mathbf{F}_\delta$ with the same $g=g(b)$. In turn, since for all $m=1,2,\dots$ we have $$\mathbf{1}_mb \in L^\infty([0,\infty[,L^r) \quad \text{ for every $d \leq r<\infty$}, \qquad \supp \mathbf{1}_m b \subset [0,m] \times B(0,m),$$ the following is true: given any $\gamma_m \downarrow 0$ we can select $\varepsilon_m \downarrow 0$ so that e.g.\,$\|\tilde{b}_m - \mathbf{1}_m b\|_{L^r([0,\infty[,L^r)} \leq \gamma_m$, and so the second term $\tilde{b}_m - \mathbf{1}_m b \in \mathbf{F}_{C_S\gamma_m^2}$ with $g \equiv 0$, see Example 2. Hence,
$$
\tilde{b}_m \in \mathbf{F}_{\delta_m} \quad \text{ with $\delta_m=(\sqrt{\delta}+\sqrt{C_S\gamma_m^2})^2$ and $g=g_\delta$},
$$
Now, selecting $c_m=\frac{\delta}{\delta_m}$ and recalling that $b_m=c_m\tilde{b}_m$, we have
$b_m \in \mathbf{F}_{\delta}$ with the same $g=g_\delta$.
\end{proof}

\medskip

\section{Proof of Proposition \ref{cl1}}

We will treat more difficult case $q>2$.
We multiply the equation in \eqref{Cauchy3} by $\rho v|v|^{q-2}$ and integrate to obtain 
\begin{align}
\frac{1}{q}\langle \rho |v(r)|^q \rangle & +  \frac{4(q-1)}{q^{2}}\int_{s}^{r} \big\langle \rho |\nabla (v|v|^{\frac{q}{2}-1})|^{2}\big\rangle dt \notag \\
& = \frac{1}{q}\langle \rho |f|^q\rangle -\int_{s}^{r} \langle \rho v|v|^{q-2},b_{m} \cdot \nabla v \rangle dt + \int_{s}^{r} \langle \rho v|v|^{q-2}, |\mathsf{f}|h \rangle dt + R_q^1, \label{v_R}
\end{align}
where $R_q^1:=-\frac{1}{q}\int_s^t \langle (\nabla \rho) \cdot \nabla |v|^q\rangle.$
From now on, the terms containing $\nabla \rho$ will be denoted by $R_q^i$. We will get rid of them, using estimate \eqref{two_est}, towards the end of the proof.
The terms to estimate are
$$
M_1 \equiv -\int_{s}^{r} \langle \rho v|v|^{q-2},b_{m} \cdot \nabla v \rangle dt  \quad \text{ and } \quad M_2 \equiv \int_{s}^{r} \langle \rho v|v|^{q-2}, |\mathsf{f}|h \rangle dt.
$$
We have, using the quadratic inequality,
\begin{align*}
M_1 
& =-\frac{2}{q}\int_{s}^{r} \langle \sqrt{\rho}\nabla |v|^{\frac{q}{2}}, b_m\sqrt{\rho}|v|^{\frac{q}{2}} \rangle dt \\
& \leq \frac{2\nu}{q} \int_{s}^{r} \langle \rho |\nabla |v|^{\frac{q}{2}}|^{2}\rangle dt +\frac{1}{2q \nu}
\int_{s}^{r}\langle  \rho |b_m|^2 |v|^q\rangle dt \quad (\nu>0) \\
& (\text{we are using $b_m \in \mathbf{F}_\delta$}) \\
& \leq \frac{2\nu}{q} \int_{s}^{r} \langle \rho |\nabla |v|^{\frac{q}{2}}|^{2}\rangle dt  +\frac{1}{2q \nu}
\big [\delta \int_{s}^{r}\langle |\nabla (\sqrt{\rho} |v|^{\frac{q}{2}})|^{2}\rangle dt +\int_{s}^{r}g_\delta(t) \langle \rho |v|^{q} \rangle dt \big] \\
& \leq \bigl(\frac{2\nu}{q}+\frac{\delta}{2q\nu} \bigr) \int_{s}^{r} \langle \rho |\nabla |v|^{\frac{q}{2}}|^{2}\rangle dt  +\frac{1}{2q \nu}
\int_{s}^{r}g_\delta(t) \langle \rho |v|^{q} \rangle dt +  \frac{\delta}{2q\nu} \big(R_q^2 + R_q^3\big),
\end{align*}
where
$
R_q^2:=\int_s^r \langle |v|^{\frac{q}{2}}\nabla \rho ,\nabla |v|^{\frac{q}{2}}\rangle dt$, $R_q^3:=\int_s^r \langle \frac{|\nabla \rho|^2}{\rho} |v|^{q}\rangle dt.
$

Next, 
\begin{align*}
M_2 &\leq \frac{\gamma}{4}\int_{s}^{r}\langle \rho |v|^{q-2}|\mathsf{f}|^{2}|h|^{2} \rangle dt +\frac{1}{\gamma}\int_{s}^{r}\langle \rho |v|^{q} \rangle  dt \qquad (\gamma>0)\\
&\text{(we are applying in the first term $ac \leq a^{p}/p +c^{p'}/p'$ with $p=\frac{q}{q-2}$, $p'=\frac{q}{2}$)} \\
&\leq \frac{\gamma(q-2)}{4q}\int_{s}^{r}\langle \rho |\mathsf{f}|^{2}|v|^{q} \rangle dt +\frac{\gamma}{2q}\int_{s}^{r}\langle\rho |\mathsf{f}|^2|h|^{q} \rangle dt +\frac{1}{\gamma}\int_{s}^{r}\langle \rho |v|^{q} \rangle dt\\
& (\text{we are using $\mathsf{f} \in \mathbf{F}_\beta$}) \\
&\leq \frac{\gamma(q-2)}{4q} \bigg [\beta \int_{s}^{r}\langle \rho |\nabla |v|^{\frac{q}{2}}|^{2}\rangle dt +\int_{s}^{r}g_\beta(t) \langle \rho |v|^{q} \rangle dt \bigg] \\
&  +\frac{\gamma}{2q}\int_{s}^{r}\langle \rho |\mathsf{f}|^2|h|^q \rangle dt +\frac{1}{\gamma}\int_{s}^{r}\langle \rho |v|^{q} \rangle dt + \frac{\gamma(q-2)\beta}{4q}\big(R_q^2 + R_q^3\big).
\end{align*}
Thus, we obtain from \eqref{v_R}:
\begin{align}
 \frac{1}{q}\langle \rho |v(r)|^q \rangle & + D \int_{s}^{r}\langle \rho |\nabla (v|v|^{\frac{q}{2}-1})|^{2}\rangle dt \notag \\
& \leq \frac{1}{q}\langle \rho |f|^q\rangle + \int_{s}^{r}\bigg( A_\delta g_\delta(t) + A_\beta g_\beta(t) + \frac{1}{\gamma} \bigg) \langle \rho |v|^{q} \rangle dt \notag \\
& + \frac{\gamma}{2q}\int_{s}^{r}\langle \rho |\mathsf{f}|^2|h|^q \rangle dt + R_q^1+ \biggl(\frac{\delta}{2q \nu} + \frac{\gamma(q-2)\beta}{4q} \biggr) (R_q^2 + R_q^3), \label{i_ineq}
\end{align}
where 
$$
D=\frac{4(q-1)}{q^{2}}-\frac{2}{q}\nu -\frac{\delta}{2q \nu}-\delta \frac{\gamma (q-2)}{4q},
$$
$$
A_\delta=\frac{1}{2q \nu}, \qquad A_\beta=\frac{\gamma (q-2)}{4q}.
$$
We  maximize $D$ by taking $\nu:=\frac{\sqrt{\delta}}{2}$. Then $$D=\frac{4(q-1)}{q^{2}}-\frac{2\sqrt{\delta}}{q} - \delta \frac{\gamma (q-2)}{4q}>0 \quad \Leftrightarrow \quad q>\frac{2}{2-\sqrt{\delta}} \;\; (\text{our assumption on $q$})$$ provided that $\gamma$ is chosen sufficiently small. 

Now,  applying quadratic inequality and the inequality $|\nabla \rho| \leq \theta\sqrt{\kappa}\rho$, we obtain
\begin{align*}
R_q^1+ \biggl(\frac{\delta}{2q \nu} + \frac{\gamma(q-2)\beta}{4q} \biggr) (R_q^2 + R_q^3) & \leq c(\kappa)\bigl( \int_{s}^{r}\langle \rho |\nabla |v|^{\frac{q}{2}}|^{2}\rangle dt +  \int_s^r \langle \rho |v|^{q} \rangle dt\bigr)
\end{align*}
with $c(\kappa) \downarrow 0$ as $\kappa \downarrow 0$.
Thus, 
we obtain from \eqref{i_ineq}:
\begin{align*}
\frac{1}{q}\langle \rho |v(r)|^q \rangle & + \big(D-c(\kappa)\big) \int_{s}^{r}\langle \rho |\nabla (v|v|^{\frac{q}{2}-1})|^{2}\rangle dt  \\
& \leq \frac{1}{q}\langle \rho |f|^q\rangle + \int_{s}^{r}G(t) \langle \rho |v|^{q} \rangle dt + \frac{\gamma}{2q}\int_{s}^{r}\langle \rho |\mathsf{f}|^2|h|^q \rangle dt,
\end{align*}
where $G(t):=A_\delta g_\delta(t) + A_\beta \,g_\beta(t) + \frac{1}{\gamma} + c(\kappa)$. We fix $\kappa$ so that $D-c(\kappa)>0$.

In particular, for all $s \leq r \leq T$,
$$
\frac{1}{q}\langle \rho |v(r)|^{q} \rangle \leq  \frac{1}{q}\langle \rho |f|^q\rangle + \int_s^T G(t) \langle |v|^{q} \rangle dt + \frac{\gamma}{2q}\|\mathsf{f} |h|^{\frac{q}{2}}\|_{L^2([s,T],L_\rho^2)}, 
$$
so
$$
\frac{1}{q}\| v\|^q_{L^\infty([s,T],L_\rho^q)} \leq \frac{1}{q}\langle \rho |f|^q\rangle +  \int_{s}^{T} G(t) \langle \rho |v|^{q} \rangle dt + \frac{\gamma}{2q}\|\mathsf{f} |h|^{\frac{q}{2}}\|^2_{L^2([s,T],L_\rho^2)}.
$$

Hence, re-denoting $T$ by $r$, we obtain from the last two inequalities
\begin{align*}
\frac{1}{q}\| v\|^q_{L^\infty([s,r],L_\rho^q)} & + \big(D-c(\kappa)\big) \int_{s}^{r}\langle \rho |\nabla (v|v|^{\frac{q}{2}-1})|^{2}\rangle dt  \\
& \leq \frac{2}{q}\langle \rho |f|^q\rangle + 2\int_{s}^{r}G(t) \langle \rho |v|^{q} \rangle dt + \frac{\gamma}{q}\int_{s}^{r}\langle \rho |\mathsf{f}|^2|h|^q \rangle dt.
\end{align*}
Using 
$
\int_{s}^{r} G(t) \langle \rho |v|^{q} \rangle dt \leq \int_{s}^{r} G(t)dt\;\| v\|_{L^\infty([s,r],L_\rho^q)}
$
and assuming first that $r$ is sufficiently close to $s$ so that $\frac{1}{q}-2\int_{s}^{r} G(t) dt>0$, 
we obtain from the previous inequality the required bound, i.e.
\begin{equation*}
\| v\|_{L^\infty([s,r],L_\rho^q)}^q \leq K \bigl(\|\mathsf{f} |h|^{\frac{q}{2}}\|^2_{L^2([s,r],L_\rho^2)} + \|f\|^q_{L^q_\rho}\bigr),
\end{equation*}
with
$ K:=\big(\frac{1}{q}-2\int_{s}^{r} G(t) dt\big)^{-1} \frac{\gamma+2}{q}.
$

Applying this bound repeatedly, we extend it to an arbitrary $[s,r] \subset [s,T]$ for any $T>s$. 
The proof of Proposition \ref{cl1} is completed.

\bigskip

\section{Proof of Theorem \ref{prop2_}}

Let $v$ be the solution of (\ref{Cauchy3}). We will prove existence of generic constants $C_0$ and $\kappa$ ($\kappa$ from the definition of $\rho$) such that
\begin{align}
\|\nabla v\|^q_{L^\infty([s,r],L_\rho^q)} & + \|\nabla|\nabla v|^{\frac{q}{2}}\|^2_{L^2([s,r],L_\rho^2)}
\notag \\
& \leq C_0\big(\|\mathsf{f} |h|^{\frac{q}{2}}\|^2_{L^2([s,r],L_\rho^2)} + \|\nabla f\|_{L^q_\rho}^q\big) \label{req_ineq4}
\end{align}
for $0 \leq s \leq r \leq T$.

(\ref{req_ineq4}) combined with Proposition \ref{cl1} will yield the required.

Set
$$w:=\nabla v, \quad \nabla_i:=\partial_{x_i}, \quad w_i:=\nabla_i v, \quad w_{ik}:=\nabla_i \nabla_k v,$$
$$
\varphi_{i}:=-\nabla_i(\rho\,w_i|w|^{q-2}) \quad (1 \leq i \leq d).
$$
Put
$$
I_{q}:=\int_{s}^{r} \big \langle \rho |w|^{q-2} \sum_{i=1}^{d} |\nabla w_i |^{2} \big \rangle dt , \quad J_{q}:= \int_{s}^{r} \langle \rho |w|^{q-2}|\nabla | w| |^{2}\rangle dt.
$$ 
Multiplying the equation \eqref{Cauchy3} by the ``test functions'' $\varphi_{i}$, integrating, and summing up in $1 \leq i \leq d$, we obtain 
\begin{align*}
\sum_{i=1}^{d}\int_{s}^{r} \langle \varphi_{i}, \partial_t v \rangle dt & =\sum_{i=1}^{d}\int_{s}^{r} \langle \varphi_{i},  \Delta v \rangle dt-\sum_{i=1}^{d}\int_{s}^{r} \langle \varphi_{i}, b_m \cdot w \rangle dt + \sum_{i=1}^{d}\int_{s}^{r} \langle \varphi_{i},|\mathsf{f}|h \rangle dt
\end{align*}
or
$$
S=S_1+S_2+S_3.
$$
Our goal is to evaluate $S$, $S_1$ and to estimate $S_2$, $S_3$ in terms of $J_q$ and $I_q$ ($\geq J_q$), arriving at the \textit{principal inequality}
$$
\langle \rho |w(r)|^{q} \rangle + C_1 J_q \leq \langle  \rho |\nabla f|^q \rangle + \int_s^r G(t)\langle \rho |w|^q \rangle dt + C\int_{s}^{r}\langle \rho |\mathsf{f}|^2 |h|^q \rangle  dt
$$
with generic constants $C_1>0$, $C$ and a function $G \in L^1_{\loc}([0,\infty[)$,
from which \eqref{req_ineq4} will follow easily.

1.~Integrating by parts, we obtain:
\begin{align*}
S & =\sum_{i=1}^{d}\int_{s}^{r}  \langle \rho (w_i|w|^{q-2}), \partial_t w_i \rangle dt  \\
& =\frac{1}{q}\int_{s}^{r} \partial_t \langle \rho |w|^{q} \rangle dt  =\frac{1}{q}\langle \rho |w(r)|^{q} \rangle  - \frac{1}{q}\langle  \rho |\nabla f|^q \rangle.
\end{align*}

2.~Next, we  integrate by parts twice:
\begin{align*}
S_{1}&=-\sum_{i=1}^{d}\int_{s}^{r} \langle \nabla_i (\rho w_i|w|^{q-2}),  \Delta v \rangle dt \\
& =-\sum_{i=1}^{d}\int_{s}^{r} \langle \rho \nabla (w_i|w|^{q-2}), \nabla w_i \rangle dt + R_q^4 \\
&=-\int_{s}^{r}\langle \rho |w|^{q-2}\sum_{i=1}^{d} |\nabla w_i |^{2} \rangle dt -\frac{1}{2}\int_{s}^{r}\langle \rho \nabla |w|^{q-2},\nabla |w|^{2} \rangle dt  + R_q^4,
\end{align*}
where 
$
R_q^4:=-\sum_{i=1}^d \int_s^r \langle (\nabla \rho) |w|^{q-2}, w_i\nabla w_i \rangle  dt.
$ The terms containing $\nabla \rho$ will again be denoted by $R_q^i$. We will get rid of them towards the end of the proof.

Thus,
$$
S_1=-I_{q}-(q-2)J_{q} + R_q^4.
$$

3.~In order to estimate $S_2$, we evaluate:
\begin{equation}
\label{phi_repr}
\varphi_i=-\rho w_{ii}|w|^{q-2}-(q-2)\rho |w|^{q-3} w_i \nabla_i |w| - (\nabla_i \rho)w_i|w|^{q-2}.
\end{equation}
Thus,
\begin{align*}
S_{2} =\int_{s}^{r} \langle \rho |w|^{q-2} \Delta v ,b_m \cdot w \rangle dt & +\int_{s}^{r} \langle \rho w \cdot \nabla |w|^{q-2} ,b_m \cdot w \rangle dt + R_q^5 \\
& =: W_1+ W_2+R_q^5,
\end{align*}
where 
$
R_q^5:=-\int_s^r \langle (\nabla \rho) \cdot w|w|^{q-2},b_m \cdot w \rangle dt.
$

Let us estimate $W_1$, $W_2$. Applying the quadratic inequality, we have
\begin{align}
W_1 &  \leq \frac{\gamma}{4} \int_{s}^{r} \langle \rho |w|^{q-2}|\Delta v |^{2}\rangle dt+\frac{1}{\gamma}\int_{s}^{r} \langle \rho |b_m|^2 |w|^q\rangle dt \qquad (\gamma > 0) \notag  \\
& \text{(we are applying $b \in \mathbf{F}_\delta$)} \notag  \\
&\leq\frac{\gamma}{4} \int_{s}^{r} \langle \rho |w|^{q-2}|\Delta v |^{2}\rangle dt \notag  \\
& +\frac{1}{\gamma}
\bigg [ \delta \int_{s}^{r} \big\langle |\nabla (\sqrt{\rho}|w|^{q/2})|^{2}\big\rangle dt +\int_{s}^{r} g_\delta(t) \langle \rho |w|^{q}\rangle dt \bigg ] \notag \\
&= \frac{\gamma}{4} \int_{s}^{r} \langle \rho |w|^{q-2}|\Delta v |^{2}\rangle dt  +\frac{1}{\gamma}
\bigg [\delta\frac{q^{2}}{4}J_{q}+\int_{s}^{r} g_\delta(t) \langle \rho |w|^{q}\rangle dt \bigg ] + \frac{\delta}{\gamma}(R_q^6 + R_q^7), \label{ineq_9}
\end{align}
where 
$
R_q^6:=\int_s^r \big\langle |w|^{\frac{q}{2}}\nabla \rho ,\nabla |w|^{\frac{q}{2}}\big\rangle dt$, $R_q^7:=\int_s^r \big\langle \frac{|\nabla \rho|^2}{\rho} |w|^{q}\big\rangle dt.
$
Now, representing $ |\Delta v |^{2}= |\nabla \cdot w |^{2}$ and integrating by parts twice, we obtain
\begin{align*}
\int_{s}^{r} \langle \rho |w|^{q-2} |\Delta v |^{2}\rangle dt =-\int_{s}^{r} \langle \rho \nabla |w|^{q-2}\cdot w, \Delta v \rangle dt  & +\sum_{i=1}^{d}\int_{s}^{r} \langle \rho w_i \nabla |w|^{q-2}, \nabla w_i \rangle dt +I_{q} + R_q^8 + R_q^9\\
&=:F+H+I_{q} + R_q^8 + R_q^9,
\end{align*}
where
$
R_q^8:=-\int_s^r \langle (\nabla \rho) \cdot w|w|^{q-2},\Delta v \rangle dt$, $R_q^9:=\sum_{i=1}^d \int_s^r \langle w_i|w|^{q-2}\nabla \rho,\nabla w_i\rangle dt. 
$
Applying  the quadratic inequality, we have
$$
F \leq  (q-2)\bigg [\frac{1}{4\varkappa} \int_{s}^{r} \langle   \rho |w|^{q-2}|\Delta v|^{2} \rangle dt +\varkappa J_{q} \bigg ] \quad (\varkappa>0), \qquad 
H \leq (q-2)\bigg(\frac{1}{2}I_{q}+ \frac{1}{2}J_{q} \bigg ).
$$
Thus, for any $\varkappa > \frac{q-2}{4}$ (we will fix $\varkappa$ later),
\begin{equation}
\label{E0}
\bigg(1-\frac{q-2}{4\varkappa}\bigg)\int_{s}^{r} \langle \rho |w|^{q-2} |\Delta v |^{2} \rangle dt \leq I_{q}+(q-2) \bigg (\varkappa J_{q}+\frac{1}{2}I_{q}+ \frac{1}{2}J_{q} \bigg ) + R_q^8 + R_q^9.
\end{equation}
Thus, applying \eqref{E0} in \eqref{ineq_9}, we arrive at
\begin{align*}
W_1  \leq \frac{\gamma}{4}\frac{4\varkappa}{4\varkappa-q+2}\bigg ( I_{q}+(q-2)(\varkappa J_{q}+\frac{1}{2}I_{q}+ \frac{1}{2}J_{q}) \bigg ) &  +\frac{1}{\gamma}\bigg (\delta \frac{q^{2}}{4}J_{q}+\int_{s}^{r}g_\delta(t)\langle \rho|w|^{q} \rangle dt \bigg ) \\
& + \frac{\gamma}{4}\frac{4\varkappa}{4\varkappa-q+2} (R_q^8 + R_q^9) + \frac{\delta}{\gamma}(R_q^6 + R_q^7).
\end{align*}

\begin{remark}
\label{el_par}
In the elliptic case one can estimate $S_2$ more efficiently, using the equation for $v$ one more time to evaluate $\Delta v$, see \cite{KS}.
\end{remark}

Next, using the quadratic inequality, we obtain
\begin{align*}
W_2 &\leq (q-2) \int_{s}^{r} \langle \rho |w|^{q-2} \big|\nabla |w|\big|,|b_m||w|^{\frac{q}{2}}\rangle dt\\
&\leq (q-2)\bigg (\nu \int_{s}^{r}\langle\rho |w|^{q-2}|\nabla |w||^{2}\rangle dt +\frac{1}{4 \nu}\int_{s}^{r}\langle \rho |b_m|^2|w|^q\rangle dt \bigg) \qquad (\nu>0) \\
& \text{(we are applying $b \in \mathbf{F}_\delta$)}\\
&\leq (q-2)\bigg (\nu J_{q}+\frac{\delta}{4 \nu} \frac{q^{2}}{4}J_{q} +\frac{1}{4 \nu}\int_{s}^{r}g_\delta(t) \langle \rho|w|^{q}\rangle dt  + \frac{\delta}{4\nu}(R_q^6+R_q^7)\bigg),
\end{align*}
completing the estimate of $S_2$.

4.~To estimate $S_3$, we represent it, using \eqref{phi_repr}, in the form
\begin{align*}
S_3 =\int_{s}^{r} \langle \rho |w|^{q-2}\Delta v,|\mathsf{f}|h \rangle dt & + \int_{s}^{r} \langle \rho w \cdot \nabla |w|^{q-2},|\mathsf{f}|h \rangle dt + R_q^{10} \\
&=: Z_{1}+Z_{2} + R_q^{10},
\end{align*}
where
$
R_q^{10}:=\int_s^r \langle (\nabla \rho) \cdot w|w|^{q-2},|\mathsf{f}|h \rangle dt.
$

Applying the quadratic inequality, we have
\begin{align*}
Z_1 & \leq  \frac{\alpha_{1}}{4}  \int_{s}^{r} \langle \rho |w|^{q-2}|\Delta v |^{2}\rangle dt +\frac{1}{\alpha_{1}} \int_{s}^{r} \big\langle \rho |w|^{q-2}|\mathsf{f}|^2 |h|^2 \big\rangle dt \qquad (\alpha_1>0) \\
& \text{(we are using (\ref{E0}))}\\
&\leq \frac{\varkappa \alpha_{1}}{4\varkappa -q+2}\big [I_{q}+(q-2)(\varkappa J_{q}+\frac{1}{2}I_{q}+\frac{1}{2}J_{q})\big ]+\frac{1}{\alpha_{1}} L_q + \frac{\varkappa \alpha_{1}}{4\varkappa -q+2}(R_q^8 + R_q^9),
\end{align*}
where $$L_q:=\int_{s}^{r} \big\langle \rho |w|^{q-2}|\mathsf{f}|^2 |h|^2 \big \rangle dt.$$
Next,
\begin{align*}
Z_2 &\leq (q-2)\int_{s}^{r}  \big \langle \rho |w|^{\frac{q}{2}-1}\nabla |w|, |w|^{(q-2)/2}|\mathsf{f}| |h| \big\rangle dt \qquad (\alpha_2>0)\\
&\leq (q-2)\bigg [ \alpha_{2}J_{q}+\frac{1}{4\alpha_{2}} L_q \bigg ].
\end{align*}
It remains to estimate $L_q$.
Applying inequality $ac \leq \varepsilon^p \frac{a^{p}}{p} +\varepsilon^{-p'}\frac{c^{p'}}{p'}$ ($\varepsilon>0$) with $p=q/(q-2)$, $p'=q/2$, we obtain
\begin{align*}
L_{q} &= \int_{s}^{r} \langle \rho^{\frac{q-2}{q}} |\mathsf{f}|^{2-\frac{4}{q}}|w|^{q-2}, \rho^{\frac{2}{q}} |\mathsf{f}|^{\frac{4}{q}}|h|^{2}\rangle dt \\
&\leq \frac{q-2}{q}\varepsilon^{\frac{q}{q-2}}\int_{s}^{r} \langle \rho |\mathsf{f}|^{2} |w|^{q} \rangle dt +\frac{2}{q}\varepsilon^{-\frac{q}{2}}\int_{s}^{r} \langle \rho |\mathsf{f}|^2 |h|^q\rangle  dt \\
& \text{(we are using $\mathsf{f} \in \mathbf{F}_\beta$)}\\
&\leq \frac{q-2}{q}\varepsilon^{\frac{q}{q-2}} \bigg [\beta \frac{q^{2}}{4}J_{q}+\int_{s}^{r} g_\beta(t) \langle \rho |w|^{q}\rangle dt  + \beta (R_q^6 + R_q^7)\bigg ] +\frac{2}{q}\varepsilon^{-\frac{q}{2}}\int_{s}^{r} \big\langle \rho |\mathsf{f}|^2 |h|^q \big\rangle  dt.
\end{align*}
This completes the estimate of $S_3$.

\medskip

5.~Applying 1-4 in the identity $S=S_1 +S_2 +S_3$, we obtain
\begin{align}
\label{E1}
\frac{1}{q}\langle \rho |w(r)|^{q} \rangle+ N I_{q}+M J_{q} &\leq \frac{1}{q}\langle |\nabla f|^q\rangle +   \int_{s}^{r}\bigl[ A_\delta g_\delta(t) + A_\beta g_\beta\bigr] \langle \rho |w|^{q}\rangle dt \\ \notag
& + C\int_{s}^{r}\langle \rho |\mathsf{f}|^2 |h|^q \rangle  dt + \sum_{i=4}^{10} c_i R_q^i
\end{align}
for appropriate generic constants $c_i>0$,
where 
$$
N:=1-\frac{(\gamma+\alpha_{1})\varkappa}{4\varkappa -q+2}\bigg(1+\frac{1}{2}(q-2)\bigg),
$$
\begin{align*}
M:=q-2 -(q-2)&\bigg(\nu+\alpha_{2}+\frac{\delta}{16\nu}q^{2} +\frac{q\beta \varepsilon^{\frac{q}{q-2}} }{4}(\frac{1}{\alpha_{1}}+\frac{q-2}{4\alpha_{2}})   \bigg) \\
& -\frac{\delta}{\gamma}\frac{q^{2}}{4}-\frac{(\gamma+\alpha_{1})\varkappa}{4\varkappa -q+2}(q-2)\big(\varkappa+\frac{1}{2}\big),
\end{align*}
$$
A_\delta:= \frac{q-2}{4 \nu}
+\frac{1}{\gamma}, \qquad A_\beta:= \varepsilon^{\frac{q}{q-2}}\frac{q-2}{q}\bigg(\frac{1}{\alpha_{1}}+\frac{q-2}{4\alpha_{2}}\bigg) \qquad \text{ and } \qquad 
C:=\varepsilon^{-\frac{q}{2}}\frac{2}{q}\bigg (\frac{1}{\alpha_{1}}+\frac{q-2}{4\alpha_{2}} \bigg ).
$$
We maximize $N + M$ by choosing
$$
\nu:=\frac{q\sqrt{\delta}}{4}, \quad \varkappa:=\frac{q-1}{2}, \quad \gamma:=\frac{q\sqrt{\delta}}{q-1}.
$$
Then  $N=1-\frac{q\sqrt{\delta}}{2} + o(\varepsilon,\alpha_1,\alpha_2)$, where $ o(\varepsilon,\alpha_1,\alpha_2)$ can be made as small as needed by selecting sufficiently small $\alpha_1$, $\alpha_2$ and $\varepsilon$. Due to $\sqrt{\delta}<q^{-1}$ we may and will select $\alpha_1$, $\alpha_2$ and $\varepsilon$ so that $N>0$.
Thus, we can apply in \eqref{E1} the elementary inequality $I_{q} \geq J_{q}$ arriving at
\begin{align}
\frac{1}{q}\langle \rho|w(r)|^{q} \rangle &  +  (N+M) J_q  \leq \frac{1}{q}\langle \rho|\nabla f|^{q} \rangle \label{in0} \\
& + \int_{s}^{r}G(t) \langle \rho |w|^{q}\rangle dt    + C\int_{s}^{r}\langle \rho |\mathsf{f}|^2 |h|^q \rangle  dt + \sum_{i=4}^{10} c_i R_q^i \notag
\end{align}
for all $s \leq r \leq T$, where we denoted $G(t):=A_\delta g_\delta(t) + A_\beta g_\beta(t)$.
 
Now, $$N+ M =\big( q-1-\frac{q\sqrt{\delta}}{2}(2q-3) \big) + o_2(\varepsilon,\alpha_1,\alpha_2),$$
where $o_2(\varepsilon,\alpha_1,\alpha_2)$ can be made as small as needed by selecting sufficiently small $\alpha_1$, $\alpha_2$ and $\varepsilon$. Our assumptions on $\delta$ and $q$ ensure that we can select $\alpha_1$, $\alpha_2$ and $\varepsilon$ so that in \eqref{in0}
$
N+M>0.
$

Next, applying the quadratic inequality, the form-boundedness of $\mathsf{f}$ and the inequality $|\nabla \rho| \leq \theta\sqrt{\kappa}\rho$, we can estimate the $R_q^i$ terms in \eqref{in0} by $\int_s^r\langle \rho|w(t)|^{q} \rangle dt$, $J_q$ and $\int_{s}^{r}\langle \rho |\mathsf{f}|^2 |h|^q \rangle  dt$, with coefficient $c(\kappa)$ that can be made as small as needed by selecting $\kappa$ sufficiently small. 

Thus, \eqref{in0} yields
\begin{align*}
\frac{1}{q}\langle \rho |w(r)|^{q} \rangle & + \big(N+M-c(\kappa)\big)\int_{s}^{r} \langle \rho | \nabla |w|^{\frac{q}{2}}|^{2} \rangle dt \\ & \leq \frac{1}{q}\|\nabla f\|_{L^q_\rho}^q + \int_{s}^{r}\big(G(t)+c(\kappa)\big) \langle \rho |w|^{q}\rangle dt  + \big(C+c(\kappa)\big)\|\mathsf{f} |h|^{\frac{q}{2}}\|^2_{L^2([s,r],L_\rho^2)}
\end{align*}
with $\kappa$ chosen sufficiently small so that $N+M-c(\kappa)>0$.

Finally, arguing as in the end of the proof of Proposition \ref{cl1} (selecting $r$ sufficiently close to $s$ and using the reproduction property), we obtain \eqref{req_ineq4}. The proof of Theorem \ref{prop2_} is completed.

\bigskip

\section{Some corollaries of Theorem \ref{prop2_}}

\label{reg_sect}

Let $\{b_m\}$ be a bounded smooth approximation of $b$ satisfying \eqref{b_m_1}, \eqref{b_m_2}.
In the proof of Theorem \ref{thm1} we will use

\begin{corollary}
\label{prop1}
In the assumptions of Theorem \ref{prop2_}, let $v \equiv v_{m}$ be the solution to Cauchy problem
\begin{equation*}
(\partial_t - \Delta + b_m \cdot \nabla)v=|\mathsf{f}|h, \quad v(s,\cdot)=0. 
\end{equation*}
For every $q \in ]d,\delta^{-\frac{1}{2}}[$, there exists a generic constant $C_0$ such that
$$
\|v\|_{L^\infty([s,r] \times \mathbb R^d)} 
\leq C_0\|\mathsf{f} |h|^{\frac{q}{2}}\|^{\frac{2}{q}}_{L^2([s,r],L^2)}
$$
for all $0 \leq s \leq r \leq T$.
\end{corollary}

\begin{proof}
Repeating the proof of Theorem \ref{prop2_} with $\rho \equiv 1$ (in which case all terms $R_q^i$ disappear), we obtain (taking into account that the initial function $f$ here is $0$)
\begin{align*}
\|v\|_{L^\infty([s,r],L^q)}^q + \|\nabla v\|^q_{L^\infty([s,r],L^q)}  + \|\nabla|\nabla v|^{\frac{q}{2}}\|_{L^2([s,r],L^2)}^2 \leq C\|\mathsf{f} |h|^{\frac{q}{2}}\|^2_{L^2([s,r],L^2)}.
\end{align*}
Now, applying the Sobolev Embedding Theorem to the second term in the LHS, we obtain the required.
\end{proof}

\begin{corollary}
\label{cor1}
Let $d \geq 3$, $b \in \mathbf{F}_\delta$ with $\delta<d^{-2}$. Let $v=v_{k,m,n}$ be the solution to 
$$
(\partial_t - \Delta + b_m \cdot \nabla)v=|b_k-b_n|h, \quad v(s,\cdot)=0,
$$
where $h \in C([0,T], \mathcal S)$.
For every $q \in ]d,\delta^{-\frac{1}{2}}[$ there exists a generic constant $C_0$ such that
$$
\|v\|_{L^\infty([s,r] \times \mathbb R^d)} 
\leq C_0\|(b_k-b_n) |h|^{\frac{q}{2}}\|^{\frac{2}{q}}_{L^2([s,r],L^2)}
$$
for all $0 \leq s \leq r \leq T$.
\end{corollary}
\begin{proof}
$\mathsf{f}:=b_k-b_n \in \mathbf{F}_{4\delta}$ in Corollary \ref{prop1}.
\end{proof}

\begin{corollary}
\label{prop_F}
In the assumptions of Corollary \ref{cor1}, let $v \equiv v_m$ be the solution to 
$$
(\partial_t - \Delta + b_m \cdot \nabla)v=h, \quad v(s,\cdot)=0, \quad  h \in C^\infty_c([0,T] \times \mathbb R^d).
$$
For every $q \in ]d,\delta^{-\frac{1}{2}}[$ there exists a generic constant $C$ such that
\begin{align*}
\|v\|^q_{L^\infty([s,r],L^q)} + \|\nabla v\|^q_{L^\infty([s,r],L^q)}  + \|\nabla|\nabla v|^{\frac{q}{2}}\|^2_{L^2([s,r],L^2)} \leq C\|h\|^q_{L^q([s,r],L^q)},
\end{align*}
for all $0 \leq s \leq r \leq T$. 
\end{corollary}

\begin{proof}
We repeat the proof of Theorem \ref{prop2_} with  $\rho \equiv 1$, $\mathsf{f} \equiv 1$ and the initial function $f=0$.
\end{proof}

\begin{corollary}
\label{prop2}
In the assumptions of Corollary \ref{cor1}, let $v=v_{k,m}$ be the solution to
\begin{equation*}
(\partial_t - \Delta + b_m \cdot \nabla)v=|b_k|, \quad v(s,\cdot)=0.
\end{equation*}
For every $q \in ]d,\delta^{-\frac{1}{2}}[$ there exist generic constants $C_1$, $C_2$ and $\kappa$ such that
$$
\|\rho^{\frac{1}{q}} v\|_{L^\infty([s,r] \times \mathbb R^d)} 
\leq C_1\|b_k \sqrt{\rho}\|_{L^2([s,r],L^2)},
$$
$$
\|v\|_{L^\infty([s,r] \times \mathbb R^d)} 
\leq C_2\sup_{z \in \mathbb Z^d}\|b_k \sqrt{\rho_{z}}\|_{L^2([s,r],L^2)}, \quad \rho_z(x):=\rho(x-z),
$$
for all $0 \leq s \leq r \leq T$.
\end{corollary}
\begin{proof}
Repeating the proof of Theorem \ref{prop2_} with $\mathsf{f}=b_k$ and $h \equiv 1$, we obtain
$$
\|v\|^q_{L^\infty([s,r],L_\rho^q)} + \|\nabla v\|^q_{L^\infty([s,r],L_\rho^q)} \leq C \|b_k \|^2_{L^2([s,r],L_\rho^2)}.
$$
By \eqref{two_est},
\begin{align*}
\|\nabla v\|_{L^\infty([s,r],L_\rho^q)} & = \sup_{[s,r]}\langle \rho |\nabla v|^q\rangle^{\frac{1}{q}} = \sup_{[s,r]}\big\langle \big|\nabla \big(\rho^{\frac{1}{q}}v\big) - \frac{1}{q} \rho^{\frac{1}{q}-1}(\nabla \rho)v\big|^q\big\rangle^{\frac{1}{q}} \\
& \geq \|\nabla \big(\rho^{\frac{1}{q}}v\big)\|_q - \frac{1}{q}\theta\sqrt{\kappa}\|\rho^{\frac{1}{q}}v\|_q,
\end{align*}
so applying the Sobolev Embedding Theorem to the first term, we obtain the first inequality. The second inequality follows  from the first one using translations.
\end{proof}

\begin{corollary}
\label{prop_grad_f}
In the assumptions of Corollary \ref{cor1}, let $u \equiv u_m$ be the solution to 
$$
(\partial_t - \Delta + b_m \cdot \nabla)u=0, \quad u(s,\cdot)=f.
$$
For every $q \in ]d,\delta^{-\frac{1}{2}}[$ there exist generic constants $C$ and $\kappa$ such that
\begin{align*}
\|u\|^q_{L^\infty([s,r],L^q_\rho)} + \|\nabla u\|^q_{L^\infty([s,r],L^q_\rho)} & + \|\nabla |\nabla u|^{\frac{q}{2}}\|^2_{L^2([s,r],L^2_\rho)} \\
& \leq C\bigl(\|\nabla f\|^q_{L^q_\rho} + \|f\|^q_{L^q_\rho}\bigr)
\end{align*}
and 
$$
\|\rho^{\frac{1}{q}} u\|_{L^\infty([s,r] \times \mathbb R^d)} 
\leq C\bigl(\|\nabla f\|_{L^q_\rho} + \|f\|_{L^q_\rho}\bigr).
$$
for all $0 \leq s \leq r \leq T$.
\end{corollary}

\begin{proof}
The first estimate is Theorem \ref{prop2_} with $\mathsf{f}=0$. The second estimate follows from the first one using the Sobolev Embedding Theorem and arguing as in the proof of Corollary \ref{prop2}.
\end{proof}

\bigskip

\section{Proof of Theorem \ref{thm1}(\textit{i})}

\label{proof_i_sect}

The convergence result of Theorem \ref{thm1}(\textit{i}) was proved in \cite{Ki}. The method of proof
 is the parabolic variant of the iteration procedure of Kovalenko-Sem\"{e}nov \cite{KS}.
For reader's convenience we outline the proof below.

\begin{definition}
\label{def_fel}
A Feller evolution family is a family of linear operators $\{U^{t,s}\}_{0 \leq s \leq t \leq T} \subset \mathcal B(C_\infty)$ such that

1) $U^{t,r}U^{r,s}=U^{t,s}$ $(r \in [s,t])$, 

2) $U^{s,s}={\rm Id}$,  

3) $\|U^{t,s}f\|_\infty \leq \|f\|_\infty$, $U^{t,s}[C_\infty^+] \subset C_\infty^+,$

4) $
U^{r,s}=s\mbox{-}C_\infty\mbox{-}\lim_{t \downarrow r}U^{t,s} \;\;(r \geq s).
$
\end{definition}

By a standard result, if $b$ is bounded and smooth, then, given an initial function $f \in C_\infty$, there exists a unique classical solution $u$ to Cauchy problem (${\rm CP}_{b}$), and the operators
$$
U^{t,s}f(\cdot):=u(t,\cdot), \quad 0 \leq s \leq t \leq T,
$$
constitute a Feller evolution family.

\medskip

Now, let $b \in \mathbf{F}_\delta$, $\delta<d^{-2}$ and let $b_m$, $m=1,2,\dots$ be a bounded smooth  approximation of $b$ satisfying \eqref{b_m_1}, \eqref{b_m_2}. Let $U^{t,s}_m \equiv U^{t,s}(b_m)$ be the Feller evolution family determined by (${\rm CP}_{b_m}$). Let $f \in C_c^\infty$,  $u_m(t,\cdot):=U_m^{t,s}f(\cdot)$. We write $U_m=U_m^{t,s}$ if no ambiguity arises.

\medskip

\textbf{1}.~Subtracting the equations for $u_m$, $u_n$ and setting $h:=u_m-u_n$, we obtain
\begin{equation}
\label{eq_h}
\partial_t h - \Delta h + b_m \cdot \nabla h + (b_m-b_n) \cdot \nabla u_n=0, \quad h(s,\cdot)=0.
\end{equation}
Multiplying the last equation by $h|h|^{p-2}$, $p>\frac{2}{2-\sqrt{\delta}}$, integrating and
applying the quadratic inequality and $b \in \mathbf{F}_\delta$ (we only need $\delta<4$ at this step), we arrive at \cite[Lemma 2]{Ki}:

\textit{For every $0 < \alpha < 1$
there exist $\theta>0$, $k=k(\delta)>1$ and a $m_0$ such that for all $m,n \geq m_0$,
for all $p \geq p_0>\frac{2}{2-\sqrt{\delta}}$ we have
\begin{multline}
\label{iterineq}
\|u_m-u_n\|_{L^{\frac{p}{1-\alpha}}([s,s+\theta],L^{\frac{pd}{d-2+2\alpha}})}\\ \leq \left(C_0 \delta \|\nabla u_m\|^2_{L^{2\lambda'}([s,s+\theta],L^{2\sigma'}(\mathbb R^d))}\right)^{\frac{1}{p}}(p^{2k})^{\frac{1}{p}}\|u_m-u_n\|_{L^{(p-2)\lambda}([s,s+\theta],L^{(p-2)\sigma})}^{1-\frac{2}{p}},
\end{multline}
for any $\sigma$ such that $$1<\sigma<\frac{d}{d-2+2\alpha}, \quad \frac{1}{\sigma}+\frac{1}{\sigma'}=1,$$ and $$\frac{1/(1-\alpha)}{\lambda}=\frac{d/(d-2+2\alpha)}{\sigma}, \quad \frac{1}{\lambda}+\frac{1}{\lambda'}=1,$$ for a constant $C_0=C_0(\theta)$ that does not depend on $m$ or $s$.}

\medskip

\textbf{2}. Since $\delta<d^{-2}$, we can apply \cite[Lemma 1]{Ki} (or Corollary \ref{prop_grad_f} where we apply the Sobolev Embedding Theorem to the term
 $\|\nabla |\nabla u_m|^{\frac{q}{2}}\|^2_{L^2([s,r],L^2)}$) to obtain
$$
\|\nabla u_m\|^q_{L^\infty([s,r],L^q)} + \|\nabla u_m\|^q_{L^q([s,r],L^{\frac{qd}{d-2}})} \leq C \|\nabla f\|^q_q.
$$
Applying Young's inequality in the last estimate, one can bound the factor 
$\|\nabla u_m\|_{L^{2\lambda'}([s,s+\theta],L^{2\sigma'}(\mathbb R^d))}$ in inequality \eqref{iterineq} by $C_1\|\nabla f\|_q$. Now, one can iterate the resulting inequality. Set $$D_T=\{(s,t) \mid 0 \leq s \leq t \leq T\}, \qquad D_{T,\,\theta}:=D_{T} \cap \{(s,t) \mid 0 \leq t-s \leq \theta\}, \quad \theta<T.$$
We have \cite[Lemma 3]{Ki}:

\textit{In the assumptions of Theorem \ref{thm1}, 
for any $ p_0>\frac{2}{2-\sqrt{\delta}}$ there exist $\theta>0$, constants $B<\infty$ and $\gamma:=
\bigl( 1-\frac{\sigma d }{d+2}\bigr)\bigl(1-\frac{\sigma d }{d+2}+\frac{2\sigma}{p_0} \bigr)^{-1}
>0$ {\rm(}$1<\sigma<\frac{d+2}{d}${\rm)} independent of $m,n$ such that
\begin{equation}
\label{iterineq0}
\|U_mf-U_nf\|_{L^\infty(D_{T,\,\theta} \times \mathbb R^d)}  \leq B \sup_{0 \leq s \leq T-\theta}\|U_mf-U_nf\|^\gamma_{L^{p_0}([s,s+\theta], L^{p_0})} \quad \text{ for all } n,m
\end{equation}
}
(the fact that $\gamma$ is strictly positive is the main concern of the iteration procedure).

\medskip
 
\textbf{3}.~That $\{U_mf\}$ indeed converges in $L^{p_0}([s,s+\theta], L^{p_0})$ is the content of \cite[Lemma 4]{Ki}. Its proof consists of subtracting the equations for $u_m$, $u_n$ and using quadratic inequality and $b \in \mathbf{F}_\delta$ to show that the sequence $\{U_mf\}$ is fundamental in $L^\infty(D_T, L^2)$ and hence, since $U^{t,s}_m$ are $L^\infty$ contractions, in $L^\infty(D_T, L^p)$ for all $2 \leq p <\infty$.

Combining Steps 2 and 3, we obtain the convergence result of Theorem \ref{thm1}(\textit{i}) first on $D_{T,\theta}$ and then, using the reproduction property of $U_m^{t,s}$, on $D_T$:
$$
\exists\; s\mbox{-}C_\infty\mbox{-}\lim_{m \rightarrow \infty}U_m^{t,s}=:U^{t,s}f \quad \text{(loc.\,uniformly in $(s,t) \in D_T$)}, \quad f \in C_c^\infty.
$$ 
Since $U^{t,s}_m$ are (positivity preserving) $L^\infty$ contractions, so is $U^{t,s}$. Now a density argument allows to extend $U^{t,s}$ to $C_\infty$. The convergence result of Theorem \ref{thm1}(\textit{i}) follows.

\medskip

The fact that $U^{t,s}f$, $f \in C_\infty$ delivers a weak solution to \eqref{CP} is the content of \cite[Proposition 1]{Ki}.

The uniqueness of the weak solution in class $C([0,\infty[,L^2)$ was proved (in greater generality) in \cite{Se}.

\bigskip

\section{Proof of Theorem \ref{thm1}(\textit{ii})}

Having at hand Theorem \ref{thm1}(\textit{i}) and Corollaries \ref{prop1}-\ref{prop_grad_f}, we will give two proofs. Both of them use rather standard arguments.

The first proof uses the fact that we have at our disposal a backward Feller evolution family $P^{t,r}$, $0 \leq t \leq r \leq T$ (see below) and follows \cite{KiS}. 

The second proof uses a tightness argument, as in \cite{ZZ,Zh,RZ}. It yields, for every $x \in \mathbb R^d$, a weak solution $\mathbb P_x$ to \eqref{seq}. In view of the convergence result in Theorem \ref{thm1}(\textit{i}), $\mathbb P_x$, $x \in \mathbb R^d$  coincide with the probability measures determined by $P^{0,r}$.

Recall that by $X^m_t=X^m_{t,x}$, $m=1,2,\dots$ we denote the strong solution to the stochastic equation
$$
X^m_t=x-\int_0^t b_m(r,X^m_r)dr + \sqrt{2}W_t, \quad x \in \mathbb R^d.
$$

We will need

\begin{definition}
\label{def_fel2}
A backward Feller evolution family $\{P^{t,r}\}_{0 \leq t \leq r \leq T} \subset \mathcal B(C_\infty)$ satisfies 

$1^\circ$)  $P^{t,s}P^{s,r}=P^{t,r}$ ($s \in [t,r]$).

$2^\circ$) $P^{r,r}={\rm Id}$,  

$3^\circ$) $\|P^{t,r}f\|_\infty \leq \|f\|_\infty$, $P^{t,r}[C_\infty^+] \subset C_\infty^+,$

$4^\circ$) $
P^{t,r}=s\mbox{-}C_\infty\mbox{-}\lim_{t \uparrow s}P^{s,r} \;\;(r \geq s).
$

\end{definition}

The terminal-value problem
\begin{align}
\label{FP}
\tag{${\rm TP}_{b}$}
\left\{
\begin{array}{l}
\bigl(\partial_t + \Delta + b(t,x)\cdot \nabla \bigr)w=0, \qquad 0 \leq t \leq r \leq T, \notag \\[2mm]
 u(r,\cdot)=f(\cdot)
\end{array}
\right.
\end{align}
determines a backward Feller evolution family: $P^{t,r}f(\cdot):=w(t,\cdot)$.

We have
\begin{equation*}
P^{t,r}(b)=U^{T-t,T-r}(\tilde{b}), \quad \tilde{b}(t,x)=b(T-t,x),
\end{equation*}
where $U^{t,s}$ is the Feller evolution family for (${\rm CP}_{b}$).

\begin{definition} 
\label{mart_sol}
A probability measure $\mathbb P_x$, $x \in \mathbb R^d$ on $(C([0,T],\mathbb R^d),\mathcal B_t=\sigma(\omega_r \mid 0\leq r \leq t))$, where $\omega_t$ is the coordinate process, is said to be a martingale solution to  stochastic equation  \eqref{seq} if

1) $\mathbb P_x[\omega_0=x]=1$;

2) $\mathbb E_{x} \int_0^r |b(t,\omega_t)|dt<\infty$, $0<r \leq T$;

3) for every $f \in C_c^2(\mathbb R^d)$ the process
$$
r \mapsto f(\omega_r)-f(x) + \int_0^r (-\Delta f + b \cdot \nabla f)(t,\omega_t)dt
$$
is a $\mathcal B_r$-martingale under $\mathbb P_x$. 
\end{definition}

\begin{definition}
\label{weak_sol}
A martingale solution $\mathbb P_x$ of \eqref{seq} is said to be a weak solution if, upon completing $\mathcal B_t$ (to $\hat{\mathcal B}_t$), there exists a Brownian motion $W_t$ on $\big(C([0,T],\mathbb R^d),\hat{\mathcal B}_t,\mathbb P_x\big)$ such that
$$
\omega_r=x - \int_0^r b(t,\omega_t)dt + \sqrt{2}W_r, \quad r \geq 0 \quad \mathbb P_x-{a.s.}
$$
\end{definition}

\subsection{Proof \textnumero 1}

\begin{claim}
\label{claim_prob}
For all $x \in \mathbb R^d$, $0 \leq t \leq r \leq T$,
$
\langle P^{t,r}(x,\cdot)\rangle=1.
$
\end{claim}

This is a consequence of Corollary \ref{prop_grad_f} and the convergence in Theorem \ref{thm1}(\textit{i}). (Indeed, working with the weight $\rho$, we can show, for a fixed $x$, that for every $\varepsilon>0$ there exists $R>0$ such that $\langle P_m^{t,r}(x,\cdot)\mathbf{1}_{\mathbb R^d - B(0,R)}(\cdot)\rangle<\varepsilon$ for all $m=1,2,\dots$, and so $\langle P_m^{t,r}(x,\cdot)\mathbf{1}_{B(0,R)}(\cdot)\rangle \geq 1- \varepsilon$. Passing to the limit in $m$ and then in $R \rightarrow \infty$, we obtain $\langle P^{t,r}(x,\cdot)\rangle \geq 1- \varepsilon$, which yields the required since $\varepsilon>0$ is arbitrary. For details, see the proof of \cite[Lemma 2]{KiS}.)

By a standard result (see e.g.\,\cite[Ch.\,2]{GvC}), given a conservative backward Feller evolution family, there exist probability measures $\mathbb P_x$ $(x \in \mathbb R^d)$ on
$(D([0,T],\mathbb R^d),\mathcal B'_t=\sigma(\omega_r \mid 0\leq r \leq t))$, where $D([0,T],\mathbb R^d)$ is the space of right-continuous functions having left limits, and $\omega_t$ is the coordinate process, such that
$$
\mathbb E_{x}[f(\omega_r)]=P^{0,r}f(x), \quad 0 \leq r \leq T. 
$$
Here and below, $\mathbb E_x:=\mathbb E_{\mathbb P_x}$.

The following estimate, which is a consequence of Corollary \ref{prop2}, plays a crucial role both in the present proof and in the alternative proof given below.

\begin{claim}
\label{claim1}
There exists a constant $C>0$ independent of $m$, $k$ such that
$$
\sup_m \sup_{x \in \mathbb R^d}\mathbb E\int_{s}^r |b_k(t,X^m_{t,x})|dt \leq CF(r-s)
$$
for $0 \leq s \leq r \leq T$, where $F(h):=h + \sup_{s \in [0,T-h]}\int_s^{s+h}g(t)dt$.
\end{claim}
\begin{proof}
Let $v=v_{m,k}$ be the solution to the terminal-value problem
$$
(\partial_t + \Delta - b_m \cdot \nabla)v=-|b_k|, \quad v(r,\cdot)=0, \quad t \leq r.
$$
By It\^{o}'s formula,
$$
v(r,X_r^m)=v(s,X_s^m) + \int_s^r (\partial_t v + \Delta v - b_m \cdot \nabla v)(t,X^m_t) dt + \sqrt{2}\int_s^r \nabla v(t,X_t^m)dW_t.
$$
Taking the expectation, we obtain
$$
\mathbb E\int_s^r |b_k(t,X_t^m)|dt = \mathbb E v(s,X_s^m).
$$
Since $\mathbb E v(s,X_s^m) \leq \|v(s,\cdot)\|_\infty$, we obtain 
from Corollary \ref{prop2}
$$
\mathbb E\int_s^r |b_k(t,X_t^m)|dt \leq \tilde{C}\sup_{z \in \mathbb Z^d}\|b_k \sqrt{\rho_{z}}\|_{L^2([s,r],L^2)}.
$$
Since $b_k \in \mathbf{F}_\delta$, we have
\begin{align*}
\|b_k \sqrt{\rho_{z}}\|^2_{L^2([s,r],L^2)} & \leq \frac{\delta}{4}\int_s^r \langle \frac{|\nabla \rho_z|^2}{\rho_z}\rangle dt + \int_s^r g(t)\langle \rho_z \rangle dt \\
& \leq \frac{\delta}{4} (r-s)\|\nabla \rho /\sqrt{\rho}\|_2^2 + \|\sqrt{\rho}\|_2^2\int_s^r g(t)dt \\
& (\text{we are using \eqref{two_est} and $\|\sqrt{\rho}\|_2<\infty$}) \\
&  \leq CF(r-s)
\end{align*}
for $0 \leq s \leq r \leq T$.
\end{proof}

\begin{claim}
$\mathbb E_x[\int_0^r |b(t,\omega_t)|dt]<\infty$.
\end{claim}
\begin{proof}
Using Theorem \ref{thm1}(\textit{i}) we can pass to the limit $m \rightarrow \infty$ in  Claim \ref{claim1} to obtain
$$
\mathbb E_x[\int_0^r |b_k(t,\omega_t)|dt] \leq CF(r-s) < \infty.
$$
Now, Fatou's Lemma applied in $k \rightarrow \infty$ yields the required.
\end{proof}

\begin{claim}
\label{claim_mart}
For every $f \in C_c^2(\mathbb R^d)$, the process
$$
M_r^f:= f(\omega_r)-f(x) + \int_0^r (-\Delta f + b \cdot \nabla f)(t,\omega_t)dt
$$
is a $\mathcal B'_r$-martingale under $\mathbb P_x$. 
\end{claim}

\begin{proof} Set $\mathbb P_x^m:=(\mathbb PX^m)^{-1}$ and $\mathbb E^m_x:=\mathbb E_{\mathbb P^m_x}$.
First, we note that 
\begin{equation}
\label{conv0}
\tag{$\star$}
\mathbb E^m_x[f(\omega_r)] \rightarrow \mathbb E_x[f(\omega_r)], \quad \mathbb E^m_x[\int_0^r (-\Delta f)(\omega_t)dt] \rightarrow \mathbb E_x[\int_0^r (-\Delta f)(\omega_t)dt] \quad (m \rightarrow \infty),
\end{equation}
as follows from the convergence result in Theorem \ref{thm1}(\textit{i}).
Next, we note that
\begin{equation}
\label{conv1}
\tag{$\star\star$}
\mathbb E^m_x \int_0^r (b_m \cdot \nabla f)(t,\omega_t)dt \rightarrow \mathbb E_x \int_0^r (b \cdot \nabla f)(t,\omega_t)dt \quad (m \rightarrow \infty).
\end{equation}
The latter follows from:

(a)
\begin{align*}
\mathbb E^m_x\bigg|\int_0^r \big((b_m-b_n)\cdot \nabla f\big)(t,\omega_t)dr \bigg| & \leq C\|(b_m-b_n)|\nabla f|^{\frac{q}{2}}\|_{L^2([0,r],L^2)} \\
& \rightarrow 0 \quad (m,n \rightarrow \infty)
\end{align*}
since $b_m \rightarrow b$ in $L^2_{\loc}([0,\infty[ \times \mathbb R^d)$ and $f$ has compact support. The inequality is proved arguing as in the proof of Claim \ref{claim1} using Corollary \ref{cor1} instead of  Corollary \ref{prop2}.

(b) 
$$
\mathbb E^m_x\bigg[\int_0^r (b_n \cdot \nabla f)(t,\omega_t)dt  \bigg] \rightarrow \mathbb E_x\bigg[\int_0^r (b_n \cdot \nabla f)(t,\omega_t)dt  \bigg] \quad (m \rightarrow \infty).
$$

(c)
\begin{align*}
\mathbb E_x\bigg|\int_0^r \big((b-b_n)\cdot \nabla f\big)(t,\omega_t)dt \bigg| & \leq C\|(b-b_n)|\nabla f|^{\frac{q}{2}}\|_{L^2([0,r],L^2)} \\
& \rightarrow 0 \quad (n \rightarrow \infty).
\end{align*}
The proof is similar to (a) (using Corollary \ref{cor1} where we pass to the limit in $m$ and then in $k$ appealing to Fatou's Lemma).

\smallskip

We are in position to complete the proof of Claim \ref{claim_mart}. 
Since 
$$
M_{r,m}^{f}:=f(\omega_r)-f(x) + \int_0^r (-\Delta f + b_m \cdot \nabla f)(t,\omega_t)dt
$$
is a $\mathcal B'_r$-martingale under $\mathbb P^m_x$, 
$$
x \mapsto \mathbb E^m_x[f(\omega_r)] - f(x) +\mathbb E^m_x\int_0^r (-\Delta f + b_m\cdot\nabla f)(t,\omega_t)dt \quad \text{ is identically zero on } \mathbb R^d,
$$
and so by \eqref{conv0}, \eqref{conv1}
$$
x \mapsto \mathbb E_x[f(\omega_r)] - f(x) +\mathbb E_x\int_0^r (-\Delta f + b\cdot\nabla f)(t,\omega_t)dt \quad  \text{ is identically zero in } \mathbb R^d.
$$
Since $\{\mathbb P_x\}_{x \in \mathbb R^d}$ are determined by a Feller evolution family, and thus constitute a Markov process, the latter suffices (see e.g.\,the proof of \cite[Lemma 2.2]{Kr1}) to conclude that $M_r^f$ is a $\mathcal B'_r$-martingale under $\mathbb P_x$.
\end{proof}

Having at hand Claim \ref{claim_mart}, we establish

\begin{claim}
$\{\mathbb P_x\}_{x \in \mathbb R^d}$ are concentrated on 
$(C([0,T],\mathbb R^d),\mathcal B_t)$.
\end{claim}

The proof repeats the proof of \cite[Lemma 4]{KiS}.

We denote the restriction of $\mathbb P_x$ from $(D([0,T],\mathbb R^d), \mathcal B_t')$  to $(C([0,T],\mathbb R^d),\mathcal B_t)$ again by $\mathbb P_x$, and thus obtain that
for every $x \in \mathbb R^d$ and all $f \in C_c^2$ 
$$
M_r^f=f(\omega_r) - f(x) + \int_0^r (-\Delta f + b\cdot\nabla f)(t,\omega_t)dt, \quad \omega \in C([0,T],\mathbb R^d),$$
is a $\mathcal B_r$-martingale under $\mathbb P_x$.

Thus, $\mathbb P_x$ is a $\mathcal B_r$-martingale solution to \eqref{seq}. 

The proof that $M_r^f$ is also a martingale for $f(x)=x_i$ and $f(x)=x_ix_j$ is obtained by following closely \cite[proof of Lemma 6]{KiS} (we have to work again with the weight $\rho$ and Corollary \ref{prop2}). It follows that
$$
r \mapsto \omega_r-x + \int_0^r b(t,\omega_t)dt
$$
is a continuous $\mathcal B_r$-martingale having the cross-variation of a Brownian motion times $\sqrt{2}$, so, by L\'{e}vy Theorem, $\mathbb P_x$ is a weak solution to \eqref{seq}.

Finally, we note that $\mathbb P_x$ satisfies estimate \eqref{kr_est0} in view of Corollary \ref{prop1}, upon applying the convergence result of Theorem \ref{thm1}(\textit{i}) and then applying Fatou's Lemma to appropriate bounded smooth approximation $\mathsf{f}_m$ of $\mathsf{f}$ that does not increase the form-bound $\beta$ of $\mathsf{f}$ (such $\mathsf{f}_m$ can be constructed as in Section \ref{rem_b_m}).

\bigskip

\subsection{Proof \textnumero 2}

Claim \ref{claim1} will again play a crucial role.

\begin{claim}
\label{claim2}
There exists a constant $\hat{C}$ independent of $m$ such that, for every $0<\beta<1$
$$
\sup_m \sup_{x \in \mathbb R^d} \mathbb E \bigg[ \sup_{t \in [0,T], a \in [0,h]} |X^{m}_{t+a,x} - X^m_{t,x}|^{\beta} \bigg] \leq \frac{\tilde{C}^\beta}{1-\beta} \tilde{F}(h), \quad h>0,
$$
where $\tilde{F}(h)=h^{\frac{1}{2}} + F(h)$.
\end{claim}

\begin{proof}
Armed with Claim \ref{claim1}, we can repeat \cite[proof of Theorem 1.1]{RZ}.
For any stopping time $\tau \leq T$,
\begin{align*}
\mathbb E\sup_{a \in [0,h]} |X^m_{\tau+a,x}-X^m_{\tau,x}| 
 \leq \mathbb E \int_\tau^{\tau + h}|b_m(t,X^m_{t,x})|dt + \sqrt{2}\mathbb E\sup_{a \in [0,h]}|W_{\tau+a}-W_\tau|
\end{align*}
Thus, applying Claim \ref{claim1} (with $k=m$), we obtain
\begin{align*}
\mathbb E\sup_{a \in [0,h]} |X^m_{\tau+a,x}-X^m_{\tau,x}|  \leq CF(h) + C_1h^{\frac{1}{2}} \leq \tilde{C}\tilde{F}(h).
\end{align*}
Now a Chebyshev-type argument \cite[Lemma 2.7]{ZZ} yields the required.
\end{proof}

It is easily seen that $F(h) \downarrow 0$ as $h \downarrow 0$, and thus so does $\tilde{F}(h)$.
Thus, applying Chebyshev's inequality in Claim \ref{claim2}, we have for every $\epsilon>0$
$$
\lim_{h \downarrow 0}\sup_m \sup_{x \in \mathbb R^d} \mathbb P \bigg[ \sup_{t \in [0,T], a \in [0,h]} |X^{m}_{t+a,x} - X^m_{t,x}|>\epsilon \bigg]=0.
$$
It follows that $\{\mathbb P_x^m:=(\mathbb PX^m)^{-1}\}_{m=1}^\infty$ is tight. Therefore, for every $x \in \mathbb R^d$ there exists a sequence $\mathbb P_{m_k}$ and a probability measure $\mathbb P_x$ on $(C([0,T],\mathbb R^d),\mathcal B_t)$ such that
$$
\mathbb P^{m_k}_x \rightarrow \mathbb P_x \quad \text{ weakly as $k \rightarrow \infty$}.
$$
The latter and the convergence result in Theorem \ref{thm1}(\textit{i}) yield
\begin{equation}
\label{conv_p_m}
\tag{$\ast$}
\mathbb P^{m}_x \rightarrow \mathbb P_x \quad \text{ weakly as $m \rightarrow \infty$}.
\end{equation}

\begin{claim}
For every $x \in \mathbb R^d$, $\mathbb P_x$ is a martingale solution to \eqref{seq}.
\end{claim}

\begin{proof}
Put $\mathbb E_x:=\mathbb E_{\mathbb P_x}$, $\mathbb E^m_x:=\mathbb E_{\mathbb P^m_x}$.
Claim \ref{claim1} can be stated as
$$
\mathbb E^m_x\int_{s}^r |b_k(t,\omega_t)|dt \leq CF(r-s), \quad \text{ for all } x \in \mathbb R^d,\;m, k=1,2,\dots,
$$
where $C$ is independent of $m$ and $k$.
Taking the limit in $m$ and then applying Fatou's Lemma in $k$, we obtain
$$
\mathbb E_x\int_{s}^r |b(t,\omega_t)|dt\quad \text{ is finite (i.e.\,$\leq CF(r-s)$).}$$

It remains to show that for every $f \in C_c^2$
$$
M^f_r:=f(\omega_r)-f(x) + \int_0^r (-\Delta f + b \cdot \nabla f)(t,\omega_t)dt
$$
is a martingale under $\mathbb P_x$. By a standard result, it suffices to show that for every ``test function'' $\eta \in C_b(C([0,T],\mathbb R^d))$, for all  $0 \leq s < r \leq T$, $$\mathbb E_x[M_r^f \eta]=\mathbb E_x[M_s^f \eta].$$ 
We will need the following:

(a)
\begin{align*}
\mathbb E^m_x\bigg|\int_0^r \big((b_m-b_n)\cdot \nabla f\big)(t,\omega_t)dt \cdot \eta(\omega) \bigg| & \leq C\|\eta\|_\infty\|(b_m-b_n)|\nabla f|^{\frac{q}{2}}\|_{L^2([0,r],L^2)} \\
& \rightarrow 0 \quad (m,n \rightarrow \infty)
\end{align*}
since $b_m \rightarrow b$ in $L^2_{\loc}([0,\infty[ \times \mathbb R^d)$ and $f$ has compact support.
To prove the inequality in (a), we argue as in the proof of Claim \ref{claim1} but use Corollary \ref{cor1} (with $k=m$).

\medskip

(b) 
$$
\mathbb E^m_x\bigg[\int_0^r (b_n \cdot \nabla f)(t,\omega_t)dt  \cdot \eta(\omega) \bigg] \rightarrow \mathbb E_x\bigg[\int_0^r (b_n \cdot \nabla f)(t,\omega_t)dt \cdot \eta(\omega) \bigg] \quad (m \rightarrow \infty).
$$

(c)
\begin{align*}
\mathbb E_x\bigg|\int_0^r \big((b-b_n)\cdot \nabla f\big)(t,\omega_t)dt \cdot \eta(\omega) \bigg| & \leq C\|\eta\|_\infty\|(b-b_n)|\nabla f|^{\frac{q}{2}}\|_{L^2([0,r],L^2)} \\
& \rightarrow 0 \quad (n \rightarrow \infty).
\end{align*}
The proof of the inequality in (c) follows closely the proof (a) (using Corollary \ref{cor1} passing to the limit in $m$ and then in $k$ appealing to Fatou's Lemma).

\medskip

Now, having at hand (a)-(c), we complete the proof.
We have
\begin{equation}
\label{e_m}
\tag{$\ast\ast$}
\mathbb E^m_x[M_r^{f,m} \eta]=\mathbb E^m_x[M_s^{f,m} \eta], \quad m=1,2,\dots,
\end{equation}
where
$$
M^{f,m}_r:=f(\omega_r)-f(x) + \int_0^r (-\Delta f + b_m \cdot \nabla f)(t,\omega_t)dt,
$$
so it remains to pass to the limit in $m$ in \eqref{e_m}. The assertions (a)-(c) yield
\begin{equation*}
\mathbb E^m_x[\int_0^r (b_m \cdot \nabla f)(t,\omega_t)dt\,\eta(\omega)] \rightarrow \mathbb E_x[\int_0^r (b \cdot \nabla f)(t,\omega_t)dt \, \eta(\omega)],
\end{equation*}
while
 $$\mathbb E^m_x[f(\omega_r)\eta(\omega)] \rightarrow \mathbb E_x[f(\omega_r)\eta(\omega)], $$
and
$$\mathbb E^m_x[\int_0^r (-\Delta f)(\omega_t)\eta(\omega) dt] \rightarrow \mathbb E_x[\int_0^r (-\Delta f)(\omega_t)\eta(\omega)dt] \quad (m \rightarrow \infty),$$
follow from \eqref{conv_p_m}.
\end{proof}

Thus, $\mathbb P_x$ is a $\mathcal B_t$-martingale solution to \eqref{seq}. The rest repeats the end of the first proof.

\bigskip

\section{Proof of Theorem \ref{thm1}(\textit{iii})}

The first statement is immediate from assertion (\textit{i}): given two bounded smooth approximations $\{b_m\}$, $\{b_m'\}$ of $b$, by (\textit{i}) their combination $\{b_1,b_1',b_2,b_2,',\dots\}$ will 
produce a new Feller evolution family that, in turn, must coincide with the Feller evolution families produced by $\{b_m\}$ and $\{b'_m\}$.

Let us prove the second statement. We have $\mathbb P^1_x$, $\mathbb P^2_x$, two martingale solutions to \eqref{seq} that satisfy
\begin{equation}
\label{kr_est1}
\mathbb E_{x}^i\int_0^T |h(t,\omega_t)| dt \leq c\|h\|_{L^q([0,T] \times \mathbb R^d)}
\end{equation}
and
\begin{align}
\label{kr_est2}
\mathbb E_{x}^i\int_0^T |b(r,\omega_t)h(t,\omega_t)|dt  \leq c\|b|h|^{\frac{q}{2}}\|^{\frac{2}{q}}_{L^2([0,T] \times \mathbb R^d)}, \quad  h \in C([0,T], \mathcal S)
\end{align}
with constant $c$ independent of $h$  ($i=1,2$).
Here and below, $\mathbb E_{x}^1:=\mathbb E_{\mathbb P_x^1}$, $\mathbb E_{x}^2:=\mathbb E_{\mathbb P_x^2}$.
 Let us show that for every $F \in C^\infty_c(]0,T[ \times \mathbb R^d)$ we have
\begin{equation}
\label{id4}
\mathbb E_x^1[\int_0^T F(t,\omega_t)dt]=\mathbb E_x^2[\int_0^T F(t,\omega_t)dt],
\end{equation}
which will then imply $\mathbb P_x^1=\mathbb P_x^2$.

Let $u_n$ be the solution to 
\begin{equation}
\label{eq_F}
\partial_t u_n -\Delta u_n + b_n \cdot \nabla u_n=F, \quad u_n(T,\cdot)=0,
\end{equation}
where $b_n$ are as in Section \ref{rem_b_m}, i.e.\,
$$
b_n=c_n e^{\epsilon_n \Delta}\mathbf{1}_n b,
$$
where $c_n \uparrow 1$, $\mathbf{1}_n$ is the indicator of $\{(t,x) \in [0,T] \times \mathbb R^d \mid |n| \leq n, |x| \leq n, |b(t,x)| \leq n\}$ and, given any $d \leq p_1 < \infty$, we can select $\epsilon_n \downarrow 0$ sufficiently rapidly so that 
\begin{equation}
\label{bbn_conv}
\|\mathbf{1}_n b - b_n\|_{L^{p_1}([0,T],L^{p_1})} \rightarrow 0 \quad n \rightarrow \infty.
\end{equation}
Set $\tau_R:=\inf\{t \geq 0 \mid |w_t| \geq R\}$, $R>0$.
Since $u_n$ is smooth, we can apply It\^{o}'s formula, obtaining ($i=1,2$)
\begin{align}
\mathbb E_x^i u_n(T \wedge \tau_R,\omega_{T \wedge \tau_R})  & = u_n(0,x)+\mathbb E_x^i \int_0^{T \wedge \tau_R} F(t,\omega_t)dt \notag \\
& +  \mathbb E_x^i \int_0^{T \wedge \tau_R} \big[(b-b_n)\cdot \nabla u_n\big](t,\omega_t)dt. \label{e_i}
\end{align}
We represent
\begin{align*}
\label{bbn}
\mathbb E_x^i \int_0^{T \wedge \tau_R} \big[(b-b_n)\cdot \nabla u_n\big](t,\omega_t)dt & = \mathbb E_x^i \int_0^{T \wedge \tau_R} \big[(b-\mathbf{1}_n b)\cdot \nabla u_n\big](t,\omega_t)dt \\
& + \mathbb E_x^i \int_0^{T \wedge \tau_R} \big[(\mathbf{1}_nb-b_n)\cdot \nabla u_n\big](t,\omega_t)dt \\
& =: I^1_n + I_n^2.
\end{align*}

1. Let us estimate $I_n^1$.
We have
\begin{align*}
I_n^1 & \leq \mathbb E_x^i \int_0^{T \wedge \tau_R} \big[|b|(1-\mathbf{1}_n) |\nabla u_n|\big](t,\omega_t)dt \\
& (\text{we are applying \eqref{kr_est2}}) \\
& \leq c\|\mathbf{1}_{B(0,R)}|b|(1-\mathbf{1}_n)|\nabla u_n|^{\frac{q}{2}}\|^{\frac{2}{q}}_{L^2([0,T] \times \mathbb R^d)}.
\end{align*}
In turn, for a $0<\theta<1$, we have
\begin{align*}
&\|\mathbf{1}_{B(0,R)}|b|(1-\mathbf{1}_n)|\nabla u_n|^{\frac{q}{2}}\|_{L^2([0,T] \times \mathbb R^d)} \\
&\leq \|\mathbf{1}_{B(0,R)}|b|(1-\mathbf{1}_n)\|^{\theta}_{L^2([0,T] \times \mathbb R^d)} \bigl(2\|\mathbf{1}_{B(0,R)}|b||\nabla u_n|^{\frac{q}{2(1-\theta)}}\|\bigr)_{L^2([0,T] \times \mathbb R^d)}^{1-\theta}.
\end{align*}
The first multiple $\|\mathbf{1}_{B(0,R)}|b|(1-\mathbf{1}_n)\|^{\theta}_{L^2([0,T] \times \mathbb R^d)} \rightarrow 0$ as $n \rightarrow \infty$.
The second multiple is uniformly (in $n$) bounded: by $b \in \mathbf{F}_\delta$,
\begin{align*}
& \||b||\nabla u_n|^{\frac{q}{2(1-\theta)}}\|^2_{L^2([0,T] \times \mathbb R^d)} \\
& \leq \delta \int_0^T \langle |\nabla |\nabla u_n|^{\frac{q}{2(1-\theta)}}|^2 \rangle dt +  \int_0^T g(t) \langle |\nabla u_n|^\frac{q}{1-\theta}\rangle dt ,
\end{align*}
where the RHS is uniformly bounded in view of Corollary \ref{prop_F}, provided that $\theta$ is chosen sufficiently close to $0$ so that $\frac{q}{1-\theta} \in ]d,\delta^{-\frac{1}{2}}[$. Thus, $I_n^1 \rightarrow 0$ as $n \rightarrow \infty$.

2.~By \eqref{kr_est1},
\begin{align*}
I_n^2 & \leq \mathbb E_x^i \int_0^{T \wedge \tau_R} \big[|\mathbf{1}_nb-b_n||\nabla u_n|\big](s,\omega_s)ds \\
& \leq c\|\mathbf{1}_{B(0,R)}|\mathbf{1}_nb-b_n|^\frac{q}{2}|\nabla u_n|^{\frac{q}{2}}\|^{\frac{2}{q}}_{L^2([0,T] \times \mathbb R^d)},
\end{align*}
where, for every $0<\theta<1$,
\begin{align*}
& \|\mathbf{1}_{B(0,R)}|\mathbf{1}_nb-b_n|^\frac{q}{2}|\nabla u_n|^{\frac{q}{2}}\|_{L^2([0,T] \times \mathbb R^d)} \\
& \leq \|\mathbf{1}_{B(0,R)}|\mathbf{1}_nb-b_n|^{\frac{q}{2\theta}}\|^\theta_{L^2([0,T] \times \mathbb R^d)}\||\nabla u_n|^{\frac{q}{2(1-\theta)}}\|^{1-\theta}_{L^2([0,T] \times \mathbb R^d)}.
\end{align*}
The second multiple is uniformly (in $n$) bounded in view of Corollary \ref{prop_F}, provided that $\theta$ is chosen sufficiently close to $0$ so that $\frac{q}{1-\theta} \in ]d,\delta^{-\frac{1}{2}}[$. The first multiple tends to $0$ as $n \rightarrow \infty$ in view of \eqref{bbn_conv} with $p_1=\frac{q}{\theta}$.

\medskip

Combining 1 and 2, we arrive at
\begin{equation}
\label{e_conv}
\left|\mathbb E_x^i \int_0^{T \wedge \tau_R} \big[(b-b_n)\cdot \nabla u_n\big](t,\omega_t)dt \right| \rightarrow 0 \quad \text{ as } n \rightarrow \infty.
\end{equation}

Now, we will need

\begin{lemma}
\label{lem_9}
$u_n$ converge uniformly on $[0,T] \times \mathbb R^d$ to a $u \in C([0,T] \times \mathbb R^d)$.
\end{lemma}

Lemma \ref{lem_9} follows from the convergence result in Theorem \ref{thm1}(\textit{i}) and the Duhamel principle. Alternatively, one can carry out the $L^{p_0} \rightarrow L^\infty$ iteration procedure used in the proof of Theorem \ref{thm1}(\textit{i}) but for equation \eqref{eq_F}. Indeed, taking the difference between solutions $u_m$, $u_n$ to equation \eqref{eq_F} with $b=b_m$, $b=b_n$, respectively, we arrive at the same equation \eqref{eq_h}, and hence to the same iteration inequality \eqref{iterineq}. To estimate the factor containing $\nabla u_m$, we appeal to Theorem \ref{prop2_} rather than to \cite[Lemma 1]{Ki}. We arrive at \eqref{iterineq0}, which yields the required uniform convergence once we establish the convergence of $u_n$ in $L^2$. The latter follows easily by modifying the proof of \cite[Lemma 4]{Ki}. 

\smallskip

Thus, using $u_n(T,\cdot)=0$ and applying \eqref{e_conv} and Lemma \ref{lem_9} in \eqref{e_i}, we pass to the limits $n \rightarrow \infty$ and then $R \rightarrow \infty$ to obtain
$$
0 =u(0,x)+\mathbb E_x^i \int_0^{T} F(t,\omega_t)dt \quad i=1,2,
$$
which yields \eqref{id4}. The proof of Theorem \ref{thm1}(\textit{iii}) is completed.

\end{document}